\documentclass[11pt]{article}
\usepackage{arxiv}
\renewcommand{\shorttitle}{}
\renewcommand{\headeright}{DFFL Under Heterogeneity}
\renewcommand{\undertitle}{}
\RequirePackage{newtxtext}
\RequirePackage{newtxmath}
\RequirePackage[T1]{fontenc}
\RequirePackage{bm}

\usepackage{amsmath,mathtools}
\usepackage{breqn}
\allowdisplaybreaks

\usepackage{amsthm}
\usepackage{enumitem}
\usepackage{comment}
\usepackage{algorithm}
\usepackage{algpseudocode}
\usepackage{tikz}
\usepackage{rotating}
\usepackage{graphicx}
\usepackage[table]{xcolor}
\usepackage{caption}
\usepackage{subcaption}
\usepackage{booktabs}
\usepackage{tabularx}
\usepackage{longtable}
\usepackage{authblk}
\usepackage{array}
\usepackage[section]{placeins}
\usepackage[hidelinks]{hyperref}
\usepackage[nameinlink,capitalize]{cleveref}
\usepackage{bookmark}
\newtheorem{theorem}{Theorem}
\newtheorem{lemma}{Lemma}
\newtheorem{proposition}{Proposition}
\newtheorem{corollary}{Corollary}
\newtheorem{assumption}{Assumption}
\theoremstyle{definition}
\newtheorem{definition}{Definition}
\newtheorem{example}{Example}
\theoremstyle{remark}
\newtheorem{remark}{Remark}

\crefname{theorem}{Theorem}{Theorems}
\crefname{lemma}{Lemma}{Lemmas}
\crefname{proposition}{Proposition}{Propositions}
\crefname{corollary}{Corollary}{Corollaries}
\crefname{example}{Example}{Examples}
\crefname{remark}{Remark}{Remarks}
\crefname{definition}{Definition}{Definitions}
\crefname{assumption}{Assumption}{Assumptions}
\usepackage{setspace}
\DeclareMathOperator*{\argmin}{\arg\!\min}
\newcommand{\Same}{\textcolor{gray}{\footnotesize [FedAvg]}}
\newcommand{\Diff}{\textcolor{blue}{\footnotesize [DFFL]}}
\newcolumntype{L}{>{\raggedright\arraybackslash}X}
\usepackage{natbib}
\bibpunct[, ]{(}{)}{,}{a}{}{,}%
\title{Decision-Focused Federated Learning Under Heterogeneous Objectives and Constraints}
\author[1]{Konstantinos Ziliaskopoulos}
\author[1]{Alexander Vinel}
\affil[1]{Auburn University}
\affil[ ]{\texttt{\{kzz0034, alexander.vinel\}@auburn.edu}}
\date{April 17, 2026}

\begin{document}

\maketitle

\begin{abstract}
We consider Decision-Focused Federated Learning (DFFL), a predict–then–optimize setting in which multiple clients collaboratively train predictive models for downstream linear optimization problems without exchanging raw data. Besides the data heterogeneity typical of standard federated learning, clients may also have different objective functions and feasible regions. Building on the SPO+ surrogate loss, we derive heterogeneity bounds that separate objective shift, measured through cost-vector distances, from feasible-set shift, measured through support-function and shape-distance terms. We show that, for general compact feasible sets, small objective perturbations can still induce nonvanishing decision-focused loss discrepancies, while strongly convex feasible regions yield sharper stability-based bounds. We then lift these pointwise bounds to a local-versus-federated excess-risk comparison, showing that federation is beneficial when the statistical advantage of pooling exceeds a client-specific heterogeneity penalty. Computational experiments on polyhedral and strongly convex problems confirm that federation is substantially more robust under strongly convex feasible regions. Finally, we evaluate a simple validation-based interpolation between local and federated DFFL models. This interpolation mitigates the theoretical tradeoff and reduces aggregate regret and worst-client harm in both synthetic experiments and a PJM energy-pricing case study.
\end{abstract}

\section{Introduction}
\label{sec:intro}

\emph{Predict--then--optimize} is a natural approach often used in modern operations problems: a predictive model maps observable context to parameters of a downstream optimization problem, and an optimization algorithm determines the decision \citep{mandi2024decision}. Specifically, given context, i.e., features or covariates, $x$ (e.g., customer, market, or system state), a model produces a cost vector $\hat c=\hat c(x)$ that is then employed in a decision problem of the form $\min_{w\in S} c^\top w$, where $w$ represents the decision (e.g., allocation, schedule, portfolio weights, order quantities) and $S$ is the feasible region \citep{ban2019big}. In such settings, improving prediction accuracy (say, mean squared error, or MSE) does not necessarily improve the quality of the {decision} induced by the prediction $\hat c$, because the optimization layer can amplify or dampen the effect of the error. \emph{Decision-focused learning (DFL)} addresses this mismatch by training models using loss functions that reflect the downstream decision quality. In particular, the Smart Predict-then-Optimize (SPO) framework and its surrogate loss function, SPO+, are the most prominent existing DFL approaches \citep{elmachtoub2022smart}.

At the same time, the data needed to learn these predictive models are often distributed across organizations or locations that cannot pool raw observations due to privacy, regulatory, or proprietary constraints. \emph{Federated learning (FL)} offers a practical mechanism for collaborative training by exchanging model updates rather than data \citep{kairouz2021advances}. Data heterogeneity is a well-known challenge in FL. In an operations context, though, FL faces another critical obstacle: \emph{clients can be heterogeneous in their data distributions and in their downstream optimization problems}. In cross-silo applications, different sites can face distinct objective distributions (e.g., different cost structures or demand profiles) and distinct feasible regions (e.g., local capacity, policy, or regulatory constraints). For example, in healthcare, treatment initiation decision options have been shown to vary across institutions \citep{ayer2019prioritizing}. In utilities, energy prices are localized due to marginal pricing and demand \citep{nerc0114cip}. In supply chains, many decision-makers are interested in downstream demand and have varying constraints on budget and resource allocation. A single global DFL model therefore trades off two forces: pooling data reduces statistical error, but optimization heterogeneity can induce {objective} and {constraint misalignment}, potentially harming decision quality for some/all clients.

This paper studies \emph{Decision Focused Federated Learning (DFFL)} through the lens of SPO+ risk to determine the impact of client heterogeneity on the federated model performance. In particular, we focus on deriving bounds that (i) quantify how decision-focused losses change when the cost vector distribution shifts across clients (objective shift), (ii) quantify how losses change when feasible sets differ (feasible-set shift), and (iii) translate these heterogeneity effects into a comparison between local and federated guarantees. Existing efforts in federated learning predominantly focus on predictive losses and treat heterogeneity as a distribution shift in data only, and generally do not directly capture the additional sensitivity created by an optimization layer, nor do they fully account for heterogeneous feasibility constraints. Conversely, decision-focused learning algorithms are developed for a single decision problem (a single pair of $c,S$) and do not directly explain the effects of federating across clients with different downstream problems.

We develop bounds on SPO+ loss heterogeneity induced by objective and feasible-set perturbations by exploiting a support-function representation of the downstream optimization problem. This viewpoint lets us treat objective and feasible-set perturbations in a unified way: objective shifts are represented by the distance between cost vectors, and feasible-set shifts are quantified by the shape distances between feasible regions. We then lift these point-wise heterogeneity bounds to the federated setting by defining a \emph{client--mixture discrepancy} term that measures how far a client's population risk can deviate from the pooled (mixture) risk. Combining this heterogeneity penalty with standard uniform convergence results yields a simple federation-gain condition: federation is likely to be helpful for a client when the bound on heterogeneity penalty is smaller than the statistical advantage of training on pooled data. This comparison also motivates interpolating between local and federated models, based on decision loss, to accommodate a heterogeneous federation.

To the best of our knowledge, there is no existing literature on decision-focused federated learning. Therefore, our main goal is to bridge the two fields, introduce the DFFL modeling framework, and characterize the main heterogeneity axes that affect it. The main contribution of the paper is as follows.
\begin{itemize}
  \item \textbf{Heterogeneity bounds for SPO+ under objective and feasible-set shift.} For the general case, we decompose the point-wise SPO+ loss difference and provide the worst-case bounds on optimization heterogeneity.
  
  \item \textbf{Sharper bounds under strongly convex feasible sets via optimizer stability.}
  For strongly convex feasible regions, we use curvature to control the stability of support points (and hence the optimizer mapping). This yields tighter bounds than is possible for general polyhedral feasible sets.

  \item \textbf{Federation-gain guarantees that isolate a heterogeneity penalty.}
  We introduce a client--mixture discrepancy quantity that captures how misaligned a client's decision-focused risk is with the pooled objective and show how to bound it from above using the proposed heterogeneity bounds. We then derive local-versus-federated excess-risk bounds for empirical risk minimization (ERM) under SPO+.


  \item \textbf{Computational evidence and decision-loss-based interpolation.} 
   We implement FedAvg-DFFL and evaluate it across polyhedral and strongly convex downstream optimization problems under controlled objective, constraint, and sample-size heterogeneity. The experiments are consistent with the theory: polyhedral problems are more sensitive to heterogeneity, while strongly convex problems are more robust to federation. We further evaluate Interp-DFFL, an implementation of a local-vs-federated interpolation rule, and show that it reduces aggregate regret and worst-client harm relative to pure local or federated versions.

\end{itemize}

The theoretical results apply to linear objective decision problems with compact feasible sets and i.i.d. sampling within each client. The federation-gain results compare local ERM to ERM on the pooled (mixture) objective. In other words, they are intended to characterize the \emph{statistical} tradeoff between pooling and heterogeneity rather than the convergence dynamics of a particular communication protocol. Our heterogeneity measures are deliberately interpretable: norm distances for the objective shifts and shape distances for the feasible set shifts. However, as with any worst-case guarantee, they may be conservative when the instance structure is favorable. Finally, while our experiments use a FedAvg-style implementation \citep{mcmahan2017communication}, our main theoretical bounds are stated at the level of risks and do not require a specific optimizer beyond producing predictors in the hypothesis class.

The remainder of the paper is organized as follows. \cref{sec:lit-review} presents the current state of the relevant literature. \cref{sec:client-hetero} develops sensitivity and heterogeneity bounds for the SPO+ loss under objective and feasible-set perturbations. \cref{sec:strongly-convex} specializes these bounds to strongly convex feasible regions and derives tighter stability-based controls. \cref{sec:fed-gain} applies the heterogeneity bounds to DFFL, introducing the client--mixture discrepancy and establishing local-versus-federated excess-risk and federation-gain guarantees. \cref{sec:experiments} reports our computational experiments across knapsack and portfolio problems under varied heterogeneity, and additionally contains our PJM case-study. \cref{sec:conclusions} concludes with implications and open directions.

\section{Related Work}
\label{sec:lit-review}

In this section, we review the most relevant results on decision-focused learning and federated learning separately. We then briefly outline the identified gap between the two research directions and position the proposed contribution.

\paragraph{Decision-Focused Learning.}
DFL (and closely related field of contextual optimization) aims to train predictors that minimize prescriptive rather than predictive error \citep{sadana2025survey, bertsimas2020predictive}. Whereas predict-then-optimize is a longstanding part of data analytics in operations management, integrating the downstream decision model into the predictive model's training is a fairly recent development in the literature \citep{mivsic2020data, kotary2021end, qi2022integrating}. 

Prescriptive analytics, where covariates inform parameters that are learned or optimized through data, is a major topic in OR literature. 
Examples include the data-driven (contextual) newsvendor with high-dimensional features and finite-sample out-of-sample cost guarantees \citep{ban2019big}, reproducing kernel Hilbert space-based data-driven optimization via objective/optimizer prediction with statistical guarantees \citep{bertsimas2022data}, and SAA frameworks that integrate residuals to approximate stochastic programs and provide asymptotic optimality and finite-sample guarantees \citep{kannan2025data}. Recent work also uses integrated estimation and optimization to obtain stable decision oracles and generalization bounds \citep{qi2025integrated}, and robust actionable prescriptive analytics under distributional ambiguity \citep{chen2026robust}.

DFL is a specific branch of prescriptive analytics \citep{mandi2024decision}, and can be categorized roughly into analytical methods \citep{halkin1974implicit, amos2017optnet}, perturbation and smoothing-based gradients \citep{domke2010implicit, poganvcic2019differentiation}, and surrogate loss functions for regret and decision quality \citep{elmachtoub2022smart,tang2024cave}.

The use of decision regret, or the optimality gap, as a loss function for training predictive models was proposed and formalized by \citet{elmachtoub2022smart}, who introduced the Smart Predict-then-Optimize framework. There, the authors formalize SPO regret and introduce SPO+ as a tractable surrogate for training. Subsequent work studied generalization, calibration, and empirical performance of SPO+ and related approaches \citep{el2019generalization,elmachtoub2020decision,jeong2022exact}. 
SPO+ is employed as the basis of our proposed DFL approach in the federated setting, since it has a strong theoretical foundation, through Fisher consistency \citep{ho2022risk} and calibration \citep{liu2021risk}, and it is known to perform well empirically \citep{mandi2020smart}. 
However, all existing DFL analyses assume a common downstream optimization problem, which may not be the case in distributed, decentralized, cross-silo operations, where the individual clients may have diverse goals, objectives, and requirements.

\paragraph{Federated Learning.}
One important constraint in cross-organizational collaboration is privacy, and it is increasingly studied in the field of operations research \citep{chen2025algorithmic}. FL enables collaborative training without sharing raw data, typically via a model parameter-averaging method performed by a server \citep{kairouz2021advances,li2025centralized,mcmahan2017communication}. Federated learning has been shown to achieve near-linear convergence speed-ups, even with multiple local training epochs and few communication rounds \citep{stich2018local}. However, while this addresses privacy concerns, it introduces a new obstacle: mitigating heterogeneity between clients. 

In federated learning literature, heterogeneity is roughly divided into three categories. (1) \textit{Data (statistical) heterogeneity:} the data between clients may be non-i.i.d, or clients may have very different sample sizes; (2) \textit{model heterogeneity:} clients may have different features, modalities, or pre-trained model architectures that need to somehow be aggregated; (3) \textit{system (device) heterogeneity:} \citep{gao2022survey, ye2023heterogeneous, su2023non}. Note, though, that in the contextual optimization setting, further heterogeneity can be important, as the decision problem itself, i.e., the objectives and constraints, can vary between different clients. 

While federated learning can still converge under moderate data heterogeneity, non-i.i.d clients significantly degrade training speed and model performance \citep{Li2020On, khaled2020tighter}. 
To combat this, there have been many algorithmic advances that build on the seminal \texttt{FedAvg} algorithm of \citep{mcmahan2017communication} to penalize weight drift \citep{li2020federated}, control gradient drift \citep{karimireddy2020scaffold}, or correct for data imbalance \citep{wang2020tackling}.
Another approach is \textit{personalization}. Here, studies train a global model through federated learning then adapt it to a specific client's domain through meta-learning approaches \citep{fallah2020personalized}, local personalization \citep{li2021ditto}, or multitask learning \citep{sattler2020clustered}. There is a similar line of work combining global and local models to personalize to individual client needs. Examples include adaptive mixtures and client-specific fine tuning \citep{m2024personalized,deng2020adaptive,mansour2020three}. These approaches adapt predictive models but do not incorporate downstream optimization layers, and therefore do not incorporate decision risk.

\paragraph{Positioning of our work.}
Heterogeneity in FL research is typically defined as a distribution shift in a supervised prediction task. However, there has been an effort to expand FL to other domains or tasks \citep{smith2017federated}. FL has been used for generative modeling \citep{augenstein2020generative}, recommendation systems \citep{flanagan2020federated}, and domain adaptation \citep{zhang2021federated}. In the scope of contextual optimization, FL has been extended to policy gradient methods \citep{khodadadian2022federated} and actor-critic methods \citep{zhuo2019federated}. 
Conceptually, our work lies at the intersection of decision-focused learning and federated learning. Whereas DFL studies how prediction errors propagate to a downstream optimization layer, and FL studies how distributed prediction models can be trained under data heterogeneity, we study how heterogeneity in the optimization layer itself affects the statistical tradeoff between local and federated training, and how that impacts the model's performance.











\section{Optimization heterogeneity between clients}
\label{sec:client-hetero}

\begin{table}[t]
\centering
\small
\renewcommand{\arraystretch}{1.12}
\setlength{\tabcolsep}{5pt}
\caption{Notation used throughout the paper.}
\label{tab:core-notation}
\begin{tabular}{@{}l p{0.73\linewidth}@{}}
\hline
\textbf{Symbol} & \textbf{Meaning} \\
\hline
$S_j$ & Feasible set for client $j$. \\

$\ell^S_{\mathrm{SPO+}}(\hat c, c)$ & SPO+ surrogate loss for feasible set $S$. \\

$\xi_S(u)$ & Support function of $S$: $\xi_S(u)=\max_{w\in S} u^\top w$. \\

$d_H(S_i,S_j)$ & Hausdorff distance between feasible sets $S_i$ and $S_j$. \\

$\delta_N$ & Translation-invariant \emph{shape distance} between feasible sets (see \cref{lem:translation-invariance}). \\

$c_d$ & Objective heterogeneity magnitude (difference between client/reference cost vectors, e.g., $\|c_i-c_j\|$). \\

$H(\hat c)$ & General heterogeneity upper bound on client SPO+ loss gap (see Theorem 1). \\

$\rho$ & Strong-convexity radius of a feasible set (used in the strongly convex refinement). \\

$\Delta_j$ & Client--mixture discrepancy: $\sup_{g\in\mathcal H}|R_j(g)-R_{\mathrm{mix}}(g)|$. \\

$R_j(g)$, $R_{\mathrm{mix}}(g)$ & Population SPO+ risk for client $j$ and the federated mixture risk, respectively. \\

$\varepsilon_j$, $\varepsilon_{\mathrm{mix}}$ & Statistical complexity / estimation terms for local and federated learning (Section 5). \\

\hline
\end{tabular}
\end{table}

\subsection{Notation and background.} 

An overview of notation used throughout this paper is given in \cref{tab:core-notation}.
We begin with preliminary definitions and assumptions that we will use throughout the paper.
Let $d\in\mathbb{N}$. For any nonempty set $S\subset\mathbb{R}^d$ define its \emph{support function}
\[
\xi_S(u) := \max_{w\in S} u^\top w \in \mathbb{R}\cup\{+\infty\}.
\]
Given a cost vector $c\in\mathbb{R}^d$ and feasible set $S$, consider the linear program
\begin{equation}
\label{eq:lp}
z_S^\star(c) := \min_{w\in S} c^\top w,
\qquad
W_S^\star(c):= \arg \min_{w\in S} c^\top w.
\end{equation}
Fix an \emph{oracle} $w_S^\star(\cdot)$ selecting a point $w_S^\star(c)\in W_S^\star(c)$ for every $c$
(whenever the set is nonempty).
The SPO loss function $\ell_{\rm SPO}(c,\hat{c})$ is defined as,
$$
\ell_{\rm SPO}(c,\hat{c}) := c^Tw^{\star}(\hat{c}) - z^{\star}(c).
$$
Following \citep{elmachtoub2022smart}, the surrogate SPO+ loss is
\begin{equation}
\label{eq:spo+}
\ell^S_{\mathrm{SPO+}}(\hat c,c)
:= \max_{w\in S}\{c^\top w - 2\hat c^\top w\} + 2\hat c^\top w_S^\star - z_S^\star,
\end{equation}
where the max term corresponds to the support function $\xi_S$, meaning the SPO+ loss can be rewritten as $\ell^S_{\mathrm{SPO+}}(\hat{c},c) = \xi_S(c-2\hat c) + 2\hat c^\top w_S^\star(c) - z_S^\star(c)$. 
The SPO+ loss is comprised of three parts: the maximum (or support function) term, the decision term, and the objective value term. Note that the objective value term itself can be expressed as a support function, i.e., $-\xi_S(-c)$.

Let $\| \cdot \|$ denote the $L^2$ norm on $\mathbb{R}^d$. While all of our results in \cref{sec:client-hetero} and some in \cref{sec:strongly-convex} apply to all norms, for simplicity's sake (and to align with our experimentation) we use the $\| \cdot \| = \| \cdot  \|_2$ throughout this paper.

For any nonempty compact sets $S_1,S_2\subset\mathbb{R}^d$, we define the Hausdorff distance 
$$d_H(S_1,S_2) := \max\Big\{ \sup_{x\in S_1}\inf_{y\in S_2}\|x-y\|,\; \sup_{y\in S_2}\inf_{x\in S_1}\|x-y\| \Big\},$$
and the diameter
$D(S):=\sup_{u,v\in S}\|u-v\|.$
Here $D(S) = D_S$ is also the Euclidean diameter of the feasible set, which matches the notation used by \citep{liu2021risk}.

Throughout the rest of the paper, we will maintain the following assumptions.
\begin{assumption}[Compactness]
\label{ass:compact}
The feasible sets of interest are nonempty and compact, so $\xi_S(\cdot)$ is finite and maxima/minima
in \eqref{eq:lp}--\eqref{eq:spo+} are attained.
\end{assumption}

\begin{assumption}[Convexity]
\label{ass:convex}
For each instance, the feasible sets of interest $S \subseteq \mathbb{R}^d$ are convex.
\end{assumption}
\begin{assumption}[Boundedness of vectors]
\label{ass:bounded}
    The objective function vector's norm $0<C_{\min}\leq \|c\| \leq C_{\max}$ is bounded above and strictly bounded away from zero, while the prediction vector's norm $\|\hat{c}\| \leq \tau$ is bounded from above.
\end{assumption}

\begin{example}[A toy example]
\label{subsec:toy-example}
Before presenting the main results, we demonstrate \emph{why} the bound in the general case has to be fairly loose, and under what assumptions we can expect a better behaved mixture model with a simplified example.

\paragraph{The polyhedral case.}
Suppose that we have two clients with coefficient vectors $c_1$ and $c_2$, with a common feasible set $S$ defined as: $S=$\{$w\in \mathbb{R}^2:0\leq w_1 \leq 1, 0\leq w_2 \leq 1, w_1 + w_2 \geq 1$\}, shown in \cref{fig:toy-polygon}, with vertices $v_1=(1,0), v_2 = (0,1),$ and $v_3=(1,1)$. For a specific $(x,c)$ pair, the clients' coefficient vectors are defined as: 
$$ c^{(1)}=c_1 + c = [1, 1+\epsilon]^T,  c^{(2)}= c_2 + c = [1+\epsilon,1]^T,$$
for $\epsilon >0$, where $c$ is the common ground-truth vector the model is trying to approximate, and $c_i$ the heterogeneity term. They each optimize $\min_w c^{(1)T}w, \min_w c^{(2)T}w$ with the corresponding optimizers $w^\star_{S}(c^{(1)}),w^\star_{S}(c^{(2)})$ at extreme points $v_1 = (1,0)$ with $c^{(1)T}v_1 = 1$, and $v_2 = (0,1)$ with $c^{(2)T}v_2 = 1$. 
Note that the change in the resulting optimal solution is, 
$$ \|w^\star_{S}(c^{(1)}) - w^\star_{S}(c^{(2)})\| = \|v_1-v_2\|= \sqrt{2} = D_S, $$
even as $\|c_1 - c_2\| = \sqrt{2} \epsilon \rightarrow0$, $\epsilon\rightarrow 0$. 

\begin{figure}[t]
    \centering
    
    \begin{subfigure}[t]{0.48\textwidth}
        \centering
        \includegraphics[width=\textwidth]{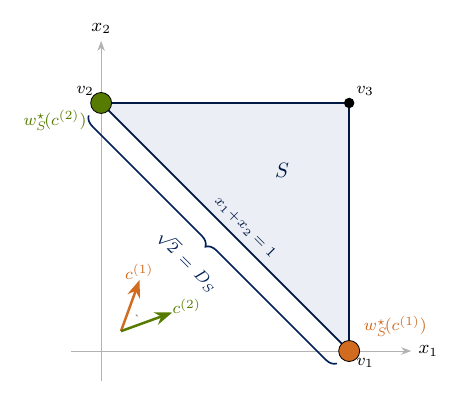}
        \caption{Polyhedral case}
        \label{fig:toy-polygon}
    \end{subfigure}
    \hfill
    \begin{subfigure}[t]{0.48\textwidth}
        \centering
        \includegraphics[width=\textwidth]{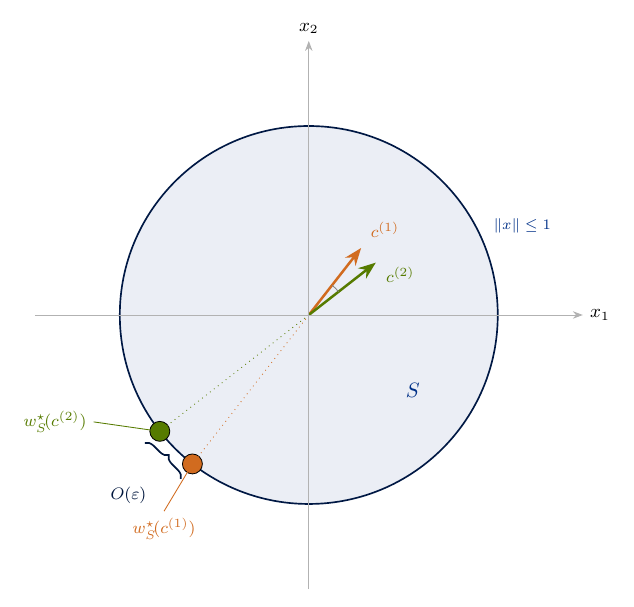}
        \caption{Strongly convex case}
        \label{fig:toy-circle}
    \end{subfigure}
    
    \caption{Example configurations of two clients with slightly different objective vectors and the impact on the downstream decision vectors.}
    \label{fig:toy-polyhedral}
\end{figure}

We can now calculate their corresponding worst-case SPO+ losses as $\max_{\hat{c}} |\ell^1_{\mathrm{SPO+}} - \ell^2_{\mathrm{SPO+}}|$. Let $\hat{c} = [\hat{c}_1, \hat{c}_2]^T,$ with $\|\hat{c}\| \leq B$, using vertex enumeration for $S$,
$$\max_{\|\hat c\| \le B} |\ell^1_{\mathrm{SPO+}}(\hat c) - \ell^2_{\mathrm{SPO+}}(\hat c)| \le \epsilon + 2\sqrt{2}B$$
In other words, an arbitrarily small difference in objective function vectors between clients can cause a large difference in pointwise SPO+ losses, a direct consequence of the fact that linear program solutions are in the extreme points. This then implies that any possible bound that we can establish on the SPO+ loss with respect to the client objective heterogeneity in a general case necessarily has to be relatively loose to account for this. A tighter bound may be possible if additional assumptions are made, though. 

\paragraph{The strongly convex case.}
Now consider the same two clients, but a unit-ball feasible set, i.e., $S=\{w\in\mathbb{R}^2:\|w\| \leq1\}$. We can observe that the optimal solution vector is less sensitive to objective shifts, and therefore so is the difference in SPO+ loss, as shown in \cref{fig:toy-circle}. Specifically, the components of the loss for $S$ are equal to: $$w^{\star}_S(c) = -\frac{c}{\|c\|}, \xi_S(u)=\|u\|,z^{\star}_S(c)=-\|c\|.$$ 
Consequently, 
\[\|w^{\star}_S(c^{(1)}) - w^{\star}_S(c^{(2)})\| = \frac{\sqrt{2}\epsilon}{\sqrt{2+\epsilon^2+2\epsilon}} = O(\epsilon).
\]

\begin{dmath}
\max_{\|\hat c\| \le B} |\ell^1_{\mathrm{SPO+}}(\hat c) - \ell^2_{\mathrm{SPO+}}(\hat c)| \leq \sqrt{2}\epsilon\left(1 + \frac{2B}{\sqrt{2+2\epsilon+\epsilon^2}}\right) = O(\epsilon(1+B)),
\end{dmath}
which vanishes as $\epsilon \rightarrow 0$, unlike the polyhedral bound $\epsilon + 2\sqrt{2}B$. The key observation here is that curvature replaces diameter in the decision term, and therefore the loss scales with cost heterogeneity, which is what we formalize in \cref{cor:thm1-sc} below. Polyhedral sets, on the other hand, force a looser bound to account for potential discontinuity in the decision vector.
\end{example}

\subsection{Bounding point-wise losses in the general case}

First, we establish the sensitivity of expression \eqref{eq:spo+}  to perturbations in the feasible set, as measured by the Hausdorff distance between the sets. 
More detailed account of the properties of the support functions can be found in \citet{schneider2013convex}. While relatively straightforward, proofs for \cref{lem:support-set,lem:zstar-set,lem:diameter-joint,lem:oracle-free,lem:translation-invariance} are given in the Appendix for the sake of completeness.
\begin{lemma}[Sensitivity to set perturbations]
\label{lem:support-set}
Let $S_1,S_2$ be nonempty compact sets and $\delta=d_H(S_1,S_2)$. Then for any $u\in\mathbb{R}^d$,
\[
|\xi_{S_1}(u)-\xi_{S_2}(u)| \le \|u\| \,\delta.
\]
\end{lemma}
\begin{lemma}[Set sensitivity of $z_S^\star$]
\label{lem:zstar-set}
Let $S_1,S_2$ be nonempty compact sets and $\delta=d_H(S_1,S_2)$. Then for any $c\in\mathbb{R}^d$,
\[
|z_{S_1}^\star(c)-z_{S_2}^\star(c)| \le \|c\| \,\delta.
\]
\end{lemma}
\begin{lemma}[Translation-invariance of the SPO+ loss function]
\label{lem:translation-invariance}
For $\ell_S(\hat c,c)$ defined as in \cref{eq:spo+} and $v,c,\hat{c} \in \mathbb{R}^d$,
$$
\ell_{S+v}(\hat c,c) = \ell_S(\hat c,c).
$$
\end{lemma}
The first two results are known support function properties. 
Intuitively, they mean that the difference in loss between two clients with different feasible sets is bounded by the distance between them, scaled by the magnitude of their objective function coefficients. 
This result can be further strengthened based on \cref{lem:translation-invariance}, which establishes set-translation invariance of SPO+. We can then consider a new distance metric, which we refer to as \textbf{shape distance}, $\delta_N := \inf_v d_H(S_{\rm ref},S_{\rm cl} + v)$. We can ``shift'' one client's feasible set in order to minimize the heterogeneity, or, in other words, the only important difference between the clients is in the shape of their feasible sets, not their distance to the origin. Two sets may have identical constraints up to a shift in coordinates, and SPO+ loss does not distinguish between them.

For the case of objective heterogeneity, we employ both the support function and objective value terms together in \cref{lem:diameter-joint}, and we can obtain an intuitive result: the terms are bounded above by the distance between objective vectors, scaled by the feasible set geometry.
\begin{lemma}[Diameter-Lipschitzness of $\xi_S(c-2\hat c)-z_S^\star(c)$]
\label{lem:diameter-joint}
Under assumptions, for any $c_1,c_2,\hat c\in\mathbb{R}^d$,
$$
\Big|
\big(\xi_S(c_1-2\hat c)-z_S^\star(c_1)\big)
-
\big(\xi_S(c_2-2\hat c)-z_S^\star(c_2)\big)
\Big|
\le D(S)\,\|c_1-c_2\|.
$$
\end{lemma}

\begin{lemma}[Bounding linear functionals across two sets]
\label{lem:oracle-free}
Let $S_1,S_2$ be compact and $\delta=d_H(S_1,S_2)$. For any $\hat c\in\mathbb{R}^d$ and any
$w_1\in S_1$, $w_2\in S_2$,
\[
|\hat c^\top w_1 - \hat c^\top w_2|
\le \|\hat c\| \big(\delta + \min\{D(S_1),D(S_2)\}\big).
\]
\end{lemma}
 As we saw in \cref{subsec:toy-example}, in the general case the decision vector can jump up to the diameter of the feasible set even under slight heterogeneity, i.e., even when $S_1 \rightarrow S_2,$ the term $D(S)$ cannot be excluded from the bound. 

We now present the main result of \cref{sec:client-hetero} in \cref{thm:main}: \textit{The difference between the point-wise SPO+ losses $\ell^i_{\mathrm{SPO+}},\ell^j_{\mathrm{SPO+}}$ of two heterogeneous clients $i,j$ in set and objective is bounded above by two terms: the decision vector shift, and the support/objective mismatch.} The proof is given in Appendix B of the online companion.

\begin{theorem}[SPO+ loss discrepancy under objective and feasible-set shift]
\label{thm:main}Define $c_d := \|c_{\rm ref} - c_{\rm cl}\|, \delta_N := \inf_v d_H(S_{\rm ref},S_{\rm cl} + v),
D_{\rm ref}:=D(S_{\rm ref}),\;\; D_{\rm cl}:=D(S_{\rm cl}),
D_{\min}:=\min\{D_{\rm ref},D_{\rm cl}\}.$

For any $\hat c\in\mathbb{R}^d$, define
\begin{dmath}
\mathcal{H}(\hat c)
= 2\|\hat c\| (\delta_N + D_{\min}) 
+\min\left\{
  D_{\rm ref}\,c_d
  + \delta_N\left(\|c_{\rm cl} - 2\hat{c}\|+\|c_{\rm cl}\|\right), 
  D_{\rm cl}\,c_d
  + \delta_N\left(\|c_{\rm ref} - 2\hat{c}\|+\|c_{\rm ref}\|\right)
\right\}.
\end{dmath}
Then
$
\bigl|\ell_{S_{\rm ref}}(\hat c,c_{\rm ref}) - \ell_{S_{\rm cl}}(\hat c,c_{\rm cl})\bigr|
\le
\mathcal{H}(\hat c).
$
\end{theorem}


\cref{thm:main} bounds the point-wise loss discrepancy by the term $\mathcal{H}(\hat{c})$, which consists of two components. The first term comes from the decision vector shift and \cref{lem:oracle-free}, and depends on (i) the difference in set shapes and (ii) the size of the smallest set. It is then scaled by the magnitude of the predictive model. 
The second term stems from the support function and the objective value terms, and it characterizes the objective heterogeneity, the feasible sets' shape distance, scaled largely by the magnitude of the objective functions and model outputs, as well as the client geometry.

We also note that the SPO+ loss is \emph{not scale invariant}. In practice, model outputs can grow arbitrarily large while gradients vanish, which would significantly weaken the bound in \cref{thm:main}. 
For the analysis, \cref{ass:bounded} prevents this scenario, and we enforce it for the experiments conducted in \cref{sec:experiments}. This means that we can essentially control the terms $\|\hat{c}\|$ and $\|c-2\hat{c}\|$ in the bound through our model design, while the other terms depend on the parameters of the clients' downstream optimization problems.

The bound in \cref{thm:main} can appear to be relatively conservative due to the diameter constant. At the same time, without additional assumptions, it cannot be eliminated, i.e., this bound is tight in general. We demonstrate this with a lower bound construction based on
\cref{subsec:toy-example}.
\begin{example}[Irreducibility of the diameter-scale term]
Consider the same polyhedral setting as in \cref{subsec:toy-example}. 
As observed, the discontinuity in the optimal solution yields a nonvanishing SPO\({+}\) discrepancy. Fix
\(B>0\) and set
$
\hat c_B=\frac{B}{\sqrt{2}}(1,-1)$ and $\|\hat c_B\|=B.$
Since \(z_S^\star(c^{(1)})=z_S^\star(c^{(2)})=1\),
$
\ell^S_{\mathrm{SPO+}}(c^{(1)},\hat c_B)
-
\ell^S_{\mathrm{SPO+}}(c^{(2)},\hat c_B)
= 
\Big[
\xi_S(c^{(1)}-2\hat c_B)
-
\xi_S(c^{(2)}-2\hat c_B)
\Big]
+
2\hat c_B^\top(w_S^\star(c^{(1)})-w_S^\star(c^{(2)})).
$
The decision term is
\[
2\hat c_B^\top(w_S^\star(c^{(1)})-w_S^\star(c^{(2)}))=2\sqrt{2}B.
\]
Moreover by \cref{lem:diameter-joint},
$
|\xi_S(c^{(1)}-2\hat c_B) -
\xi_S(c^{(2)}-2\hat c_B)|
\leq
2\epsilon.
$
Therefore,
\[
\sup_{\|\hat c\|\leq B}
\left|
\ell^S_{\mathrm{SPO+}}(c^{(1)},\hat c)
-
\ell^S_{\mathrm{SPO+}}(c^{(2)},\hat c)
\right|
\geq
2\sqrt{2}B - 2\epsilon.
\]
In this example, the feasible sets are identical, and therefore \(\delta_N=0\),
\(D_{\min}=D_S\), and
\(c_d=\|c^{(1)}-c^{(2)}\|=\sqrt{2}\epsilon\). Hence \cref{thm:main} gives
$
\mathcal H(\hat c_B)
=
2\|\hat c_B\|D_S+D_Sc_d
=
2\sqrt{2}B+2\epsilon.
$
Thus, as \(\epsilon\to 0\), the upper bound in \cref{thm:main} matches the
lower-bound construction in the leading term \(2D_SB\). Consequently,
without additional assumptions, the
diameter-scale term \(2\|\hat c\|D_{\min}\) cannot, in general, be replaced by a
quantity that vanishes with objective heterogeneity. 
An example of such an assumption is the strongly convex case considered next.
\end{example}


\section{Strongly Convex Feasible Sets}
\label{sec:strongly-convex}
\begin{definition}[Strongly convex set]
\label{def:strongly-convex}
A set $S \subseteq \mathbb{R}^n$ is called strongly convex with radius $\rho>0$ if it can be expressed as an intersection of closed balls of radius $\rho$, i.e.,
$
S = \bigcap_{x \in X} B_\rho(x),
$
where $X\subseteq \mathbb{R}^n$ is a set of centers and $B_\rho(x)$ is a closed ball of radius $\rho$ centered at $x$ \citep{polovinkin1996strongly}.
\end{definition}
\begin{assumption}[Strong convexity of the feasible set]
\label{ass:strong-convex-set}
     In addition to previous assumptions, the feasible sets of interest are strongly convex of radius $\rho$ under \cref{def:strongly-convex}.
\end{assumption}

In \cref{sec:client-hetero}, the decision term was controlled by the diameter of the set $S$, which was part of the upper bound regardless of the size of the objective perturbation. As we saw in \cref{subsec:toy-example}, strong convexity in the feasible set can make the decision term change continuously with perturbations. In this section, we formalize this in an upper bound for the strongly convex case that is tighter than the general bound of \cref{thm:main}. Proofs for \cref{lem:dir-stability,lem:QG,lem:set-stability} can be found in Appendix B of the online companion..
We also introduce the following notation to simplify the presentation. For a nonempty compact set $S \subset \mathbb{R}^d$ and a direction $p \in \mathbb{R}^d$, let
$x_S(p) \in \arg\max_{x \in S} p^Tx$
denote the maximizer of $S$ in direction $p$, and
$w_S^\star(p) \in \arg\min_{x \in S} p^Tx$
denote the minimizer in direction $p$.
\begin{lemma}[Directional stability of support points]
\label{lem:dir-stability}
Let $S \subset \mathbb{R}^d$ be nonempty, compact, and $\rho$-strongly convex. Then for every $p,q \in \mathbb{R}^{d}, \|p\|=\|q\|=1,$
the support point is unique and the map $p \mapsto x_S(p)$ is $\rho$-Lipschitz:
$
\|x_S(p) - x_S(q)\| \le \rho \, \|p-q\|
$.
Consequently, for any $c_1,c_2 \neq 0$,
\[
\|w_S^\star(c_1) - w_S^\star(c_2)\| \le \rho \left \|\frac{c_1}{\|c_1\|}-\frac{c_2}{\|c_2\|}\right\|.
\]
\end{lemma}

For the same feasible set $S$, \cref{lem:dir-stability} shows that the optimal decisions for different clients can only differ up to the difference in directions, scaled by the curvature. This is also intuitive through a geometric perspective: in the strongly convex case, the optimal point in a minimization problem is \emph{the projection of the objective vector} $-c_i$ \emph{onto the feasible set} $S$. Since optimal points cannot be found in the interior of the set, this distance cannot grow arbitrarily as in the general case. 
This is also why the normalization by the norm appears here: since the optimal point is projected back to the feasible set from the objective vector, its magnitude is irrelevant; only its direction affects the decision. Furthermore, due to \cref{ass:bounded} $\|c\| \neq 0$.

\begin{lemma}[Quadratic growth from curvature]
\label{lem:QG}
Let $S \subset \mathbb{R}^d$ be nonempty, compact, and $\rho$-strongly convex.
Fix $p \in \mathbb{R}^{d}, \|p\|=1,$ and let $x_S(p)$ be the (unique) support point $\nabla \xi_S(p) = x_S(p)$.
Then for every $x \in S$,
\[
\xi_S(p) - p^\top x \;\ge\; \frac{1}{2\rho}\,\|x - x_S(p)\|^2.
\]
\end{lemma}
\cref{lem:QG} connects suboptimal points with their objective values.  Specifically, the difference in the objective function between the optimal decision and a suboptimal one grows quadratically with the difference in decision vectors, scaled again by the curvature. \cref{lem:set-stability} uses this to bound support points in different sets under the same direction.
\begin{lemma}[Set stability of support points under Hausdorff perturbations]
\label{lem:set-stability}
Let $S_1,S_2 \subset \mathbb{R}^d$ be nonempty and compact, and let
$\delta := d_H(S_1,S_2)$ be the Hausdorff distance induced by $\|\cdot \|$.
Assume $S_i$ is $\rho_i$-strongly convex for $i=1,2$.
Then for any $p \in \mathbb{R}^{d}, \|p\| = 1$,
\[
\|x_{S_1}(p) - x_{S_2}(p) \|
\;\le\;
\delta + 2\sqrt{\rho_{\min}\,\delta},
\]
where $\rho_{\min}:=\min\{\rho_1,\rho_2\}.$
\end{lemma}

\cref{lem:set-stability} bounds support points with two terms. The first is the Hausdorff distance $\delta$ for the set mismatch, while the second term links the suboptimality of the support points for the other set to closeness in vector space through \cref{lem:QG}. Similarly to the general case, we can take advantage of the translation-invariance shown in \cref{lem:translation-invariance} to replace Hausdorff with shape distances $\delta_N$.
\sloppy
\begin{lemma}[Bounding linear functionals across two sets under strong convexity]
\label{lem:lemma4-sc}
Let $S_{\rm ref}, S_{\rm cl} \subset \mathbb{R}^d$ be nonempty and compact, and let
$\delta := d_H(S_{\rm ref}, S_{\rm cl})$ be the Hausdorff distance induced by $\|\cdot \|$.
Assume $S_{\rm ref}$ is $\rho_{\rm ref}$-strongly convex and $S_{\rm cl}$ is $\rho_{\rm cl}$-strongly convex.
For any $\hat c \in \mathbb{R}^d$,
\[
\Big|\hat c^\top\big(w^\star_{S_{\rm ref}}(c_{\rm ref}) - w^\star_{S_{\rm cl}}(c_{\rm cl})\big)\Big|
\;\le\;
\|\hat c \| \cdot \Gamma,
\]
where
$$ \Gamma = \rho_{\min} \left\|\frac{c_{\rm ref}}{\|c_{\rm ref}\|} - \frac{c_{\rm cl}}{\|c_{\rm cl}\|} \right\| + \delta + 2\sqrt{\rho_{\min}\delta}. $$
\end{lemma}

\begin{corollary}[Strongly convex feasible sets]
\label{cor:thm1-sc}
In the setting of Theorem~1, additionally assume that strong convexity of the feasible set (\cref{ass:strong-convex-set}) 
holds for $S_{\rm ref}$ and $S_{\rm cl}$ (with radii $\rho_{\rm ref},\rho_{\rm cl}$). Define $T$ as the support/objective mismatch term from \cref{thm:main}, $T :=\min\left\{
  D_{\rm ref}\,c_d
  + \delta_N\left(\|c_{\rm cl} - 2\hat{c}\|+\|c_{\rm cl}\|\right), 
  D_{\rm cl}\,c_d
  + \delta_N\left(\|c_{\rm ref} - 2\hat{c}\|+\|c_{\rm ref}\|\right)
\right\}.$
Then Theorem~1 remains valid with $H(\hat c)$ replaced by
\[
\mathcal{H}_{\rm SC}(\hat c) := B_{\rm SC}(\hat c) + T,
\]
where
\[
B_{\rm SC}(\hat c) := 2\|\hat c\| \cdot \Gamma,
\]
and
$$ \Gamma = \rho_{\min} \left\|\frac{c_{\rm ref}}{\|c_{\rm ref}\|} - \frac{c_{\rm cl}}{\|c_{\rm cl}\|}\right\| + \delta_N + 2\sqrt{\rho_{\min}\delta_N}. $$
\end{corollary}

\begin{proof}{Proof}
The proof is identical to Theorem~1, except that in Step~2 apply
\cref{lem:lemma4-sc} in place of \cref{lem:oracle-free}, and substitute Hausdorff distances with shape distances due to \cref{lem:translation-invariance}.
\end{proof}
\cref{lem:lemma4-sc} and \cref{cor:thm1-sc} replace the general decision shift term of \cref{thm:main} with a tighter version due to strong convexity. Now, as heterogeneity approaches zero, the heterogeneity term $\mathcal{H}_{\mathrm{SC}} \rightarrow 0$. These findings also align conceptually with \citet{liu2021risk}: under strong convexity, SPO+ is both well-behaved and better calibrated to SPO loss. Here, we also showed that strong convexity benefits federation, since under strong convexity, loss functions between clients degrade more regularly as heterogeneity increases.

\section{Federation Gain Bounds}
\label{sec:fed-gain}

\subsection{Derivation of the main bound}
We now apply the heterogeneity bounds developed above to analyze when federated learning improves over local learning in the decision-focused setting. We decompose the local versus global model risk difference into a statistical term (where federation helps) and a heterogeneity term (where client heterogeneity hurts). 

Let $m\in\mathbb{N}$ be the number of clients. Client $j\in\{1,\ldots,m\}$ holds $n_j$ i.i.d.\
samples from distribution $P_j$ over $\mathcal{X}\times\mathbb{R}^d$. Define the total sample
size $N:=\sum_{j=1}^m n_j$ and the mixture weights $\alpha_j := n_j/N$.
Each client $j$ has a nonempty compact feasible set $S_j\subset\mathbb{R}^d$ with Euclidean
diameter $D_j := D_{S_j}$. Let $D_{\max}:=\max_{j=1,\ldots,m} D_j$.
For the pairwise risk comparisons used in this Section, let $X \sim P_X$, where $P_X$ is the shared covariate distribution, while client-specific cost vectors $(C^{(1)},...,C^{(m)})$ are jointly distributed conditional on $X$ with marginal distributions $(X,C^{(j)}) \sim P_j$. Later in \cref{rem:optimal-transport} and \cref{sec:conclusions}, we discuss possible alternatives to this joint distribution definition.


Let $\mathcal{H}$ be a class of predictors $g:\mathcal{X}\to\mathbb{R}^d$.
We assume $\mathcal{H}$ has bounded outputs: there exists $B<\infty$ such that
$\|g(x) \| \le B$ for all $g\in\mathcal{H}$ and $x\in\mathcal{X}$.
For client $j$, define the \emph{population SPO+ risk} as
\[
R_j(g) := \mathbb{E}_{(x,c)\sim P_j}\bigl[\ell_{S_j}(g(x),c)\bigr],
\]
where $\ell_{S_j}$ is the SPO+ loss \eqref{eq:spo+} with feasible set $S_j$.
Define the \emph{federated mixture risk}:
\[
R_{\rm mix}(g) := \sum_{j=1}^m \alpha_j R_j(g),
\]
where $\alpha_j \geq 0, \sum_{j=1}^m \alpha_j=1$ are the client-mixture weights of client $j$. 
Let $\widehat{R}_j(g)$ and $\widehat{R}_{\rm mix}(g)$ denote the corresponding empirical risks.

Define the optimal predictors:
\[
g_j^\star \in \argmin_{g\in\mathcal{H}} R_j(g),
\qquad
g_{\rm mix}^\star \in \argmin_{g\in\mathcal{H}} R_{\rm mix}(g).
\]
For a class $\mathcal{H}$ of vector-valued functions $g:\mathcal{X}\to\mathbb{R}^d$ and $n$ samples,
the Rademacher complexity is
\[
\mathfrak{R}_n(\mathcal{H}) := \mathbb{E}\left[\sup_{g\in\mathcal{H}}
\frac{1}{n}\sum_{i=1}^n \sum_{k=1}^d \sigma_{ik}\, g_k(x_i)\right],
\]
where $\sigma_{ik}$ are independent Rademacher random variables.

\begin{lemma}[Uniform convergence for SPO+ risk]
\label{lem:uniform-convergence}
\citep{liu2021risk} For any distribution $P$ over $\mathcal{X}\times\mathbb{R}^d$ and feasible set $S$ with diameter $D_S$,
with probability at least $1-\delta$,
$$
\sup_{g\in\mathcal{H}} \bigl|R_S(g) - \widehat{R}_{S,n}(g)\bigr|
\;\le\;
4\sqrt{2}\,D_S\,\mathfrak{R}_n(\mathcal{H}) + b\sqrt{\frac{2\log(1/\delta)}{n}},
$$
where $b = \sup_{\hat{c} \in \mathcal{H}(\mathcal{X}), c \in C} \ell(\hat{c},c) \le D_SC_{\max} + 2 D_S \sup_{g\in \mathcal{H},x\in \mathcal{X}} \|g(x)\| $ for bounded true cost vector norm $\|c\| \leq C_{\max} < \infty.$
\end{lemma}

Define the local and federated statistical complexity terms:
\[
\varepsilon_j(n_j,\delta) := 4\sqrt{2}\,D_j\,\mathfrak{R}_{n_j}(\mathcal{H})
+ b\sqrt{\frac{2\log(1/\delta)}{n_j}},
\]
\[
\varepsilon_{\rm mix}(N,\delta) := 4\sqrt{2}\,D_{\max}\,\mathfrak{R}_N(\mathcal{H})
+ b\sqrt{\frac{2\log(1/\delta)}{N}}.
\]

The key quantity measuring heterogeneity is the following term, which we refer to as Client--mixture discrepancy.
For client $j$, define
\begin{equation}
\label{def:discrepancy}
\Delta_j := \sup_{g\in\mathcal{H}} \bigl|R_j(g) - R_{\rm mix}(g)\bigr|.
\end{equation}
This quantity captures how misaligned client $j$'s objective is with the federated mixture
objective. 

We now present the main result of this section, which in essence states that the federated model helps client $j$ if pooled data sample complexity gains $(\varepsilon_j - \varepsilon_{\mathrm{mix}})$ exceed the client’s heterogeneity penalty $\Delta_j$.
\begin{theorem}[Local-Federated Excess Risk Comparison]
\label{thm:federation-gain}
Let $\hat g_j^{\rm loc}$ be the ERM on client $j$:
\[
\hat g_j^{\rm loc} \in \argmin_{g\in\mathcal{H}} \widehat{R}_j(g),
\]
and let $\hat g^{\rm fed}$ be the ERM on the pooled mixture objective:
\[
\hat g^{\rm fed} \in \argmin_{g\in\mathcal{H}} \widehat{R}_{\rm mix}(g).
\]
With probability at least
$1-\delta$ over all data:

\textbf{(Local bound)}
\begin{equation}
\label{eq:local-bound}
R_j(\hat g_j^{\rm loc}) - R_j(g_j^\star) \;\le\; 2\,\varepsilon_j\!\left(n_j,\frac{\delta}{2}\right).
\end{equation}

\textbf{(Federated bound)}
\begin{equation}
\label{eq:fed-bound}
R_j(\hat g^{\rm fed}) - R_j(g_j^\star) \;\le\; 2\,\varepsilon_{\rm mix}\!\left(N,\frac{\delta}{2}\right) + 2\Delta_j.
\end{equation}
\end{theorem}
The proof of \cref{thm:federation-gain} follows from the ERM argument $R(\hat{g}) \leq \widehat{R}(\hat{g}) + \varepsilon_j$ (see \citet{shalev2014understanding}) and \cref{lem:uniform-convergence}. A formal proof can be found in Appendix B of the online companion. Additional connections between the SPO+ excess-risk comparison, estimation-based federation gain, and SPO regret are provided in Appendix B.

\begin{remark}[The local-federated tradeoff]
\label{cor:when-federation-helps}
\cref{thm:federation-gain} separates the local versus federated comparison into two terms: a statistical term, which favors federation through the larger pooled sample size, and a heterogeneity term, which penalizes mismatches between client $j$'s risk and the mixture risk. The federated guarantee is tighter for client $j$ whenever
\begin{equation}
\label{eq:federation-condition}
\Delta_j \;<\; \varepsilon_j(n_j,\delta/2) - \varepsilon_{\rm mix}(N,\delta/2).
\end{equation}
That is, federation is most attractive when the pooled-data statistical advantage for client $j$ exceeds the optimization heterogeneity penalty.
\end{remark}

Intuitively, $\Delta_j$ measures how 'atypical' client $j$ is compared to the pooled population. If it looks like the mixture, $\Delta_j$ is small and the pooled statistical advantage will dominate, proportional to the gain in statistical complexity.

\begin{proposition}[Bounding $\Delta_j$ via heterogeneity]
\label{prop:delta-bound}
Suppose all clients share a common feature distribution, and let
$\mathcal{H}_{ij}(\hat c)$ denote the fixed-prediction discrepancy bound from \cref{thm:main}
between clients $i$ and $j$. Then
\[
\Delta_j \;\le\; \sum_{i=1}^m \alpha_i \sup_{g\in\mathcal{H}}
\mathbb{E}_{(x,C^{(i)},C^{(j)})\sim P_{i j}}\bigl[\mathcal{H}_{ij}(g(x))\bigr].
\]

\end{proposition}
The proof can be found in Appendix B.

Because SPO+ is a surrogate for SPO regret, existing calibration results imply corresponding SPO-regret consequences, with stronger transfer under strongly convex feasible regions. We summarize these extensions in the Appendix.

\begin{remark}[Heterogeneous feature distributions via optimal transport]
\label{rem:optimal-transport}
\cref{sec:fed-gain}'s definitions specify a shared covariate distribution for clients. More generally, one could optimize over distribution couplings, and taking the infimum over all valid couplings would yield a bound in terms of an optimal-transport distance between client distributions. However, determining the appropriate conditions for $\mathcal{H}$ and $P_i$ is not simple, due to the dependency on the prediction $g(x)$ and $c^{(i)}, c^{(j)}, S_i, S_j$, but is nevertheless an interesting direction that would link DFFL to covariate-shift literature in FL, which we discuss further in \cref{sec:conclusions}.
\end{remark}

\section{Computational Experiments}
\label{sec:experiments}

\subsection{Implementation details}

We follow the synchronous \texttt{FedAvg} \citep{mcmahan2017communication} round or episode structure, and the original \texttt{FedAvg} procedure and notation (see Algorithms 1-3 in the online companion). We replace the local loss function with $\ell^{\text{SPO+}}$ and local gradients with the decision-focused update \citep{elmachtoub2022smart}. We use the Adam optimizer ($\eta = 10^{-3}$) with gradient clipping ($c = 1$) for local updates. We then aggregate client updates by sample counts, as in the original algorithm. 
For each configuration, we train both federated and local models for each client using a shallow neural network structure with ReLU activation function ($x\rightarrow8\rightarrow64\rightarrow50\rightarrow\ell_2 \text{clip}(\tau)\rightarrow \hat{c}$). We apply $\ell_2$-norm clipping on the output vector, by multiplying the model output $\hat{c}$ by $\min(1, \frac{\tau}{\|\hat{c} \| + \epsilon})$, where $\epsilon>0$ and $\tau=20$. We set the number of clients to 20.
All experiments were implemented in \texttt{Julia} \citep{bezanson2017julia}. We employ the \texttt{Lux.jl} framework to construct and train the neural networks \citep{pal2023lux} and we compute parameter gradients using the \texttt{Zygote.jl} reverse mode automatic differentiation system \citep{innes2019differentiable}. 

We perform experiments across two downstream optimization problems: a fractional knapsack model representing the polyhedral case, and an entropy-bound portfolio optimization model as a strongly convex example. 
The
\emph{fractional knapsack} follows \citet{ho2022risk}. Here, each client $j$ 
solves
$$
\min_{w \in[0,1]^m} \Big\{ (-c_j)^Tw\,  \Big|\ a^Tw\leq B_j \Big \},
$$
where $a$ are the item weights, and $B_j$ is the client's total allocated budget. 
\emph{Entropy-constrained portfolio} problem follows \citet{liu2021risk}, where it was used to evaluate the performance of the original SPO+. In this case, each client $j$ 
solves
$$
\min_{w \in \mathbb{R}^m} \Big \{c_j^Tw \, \Big | \, w \geq 0, \sum_{i=1}^m w_i = 1, \sum_{i=1}^m w_i \log w_i \leq r_j\Big \},
$$
where $r_j$ is the threshold of entropy of portfolio $w$ for client $j$. It can be seen that the formulation is strongly convex when restricted to the simplex $\sum_{i=1}^m w_i = 1$.

Both formulations allow for a direct polynomial solution, which is necessary for us to be able to perform the required scale of experiments. The fractional knapsack is solved with the greedy approach, by simply sorting the items according to profit-to-weight ratio  and adding them until reaching the allowed budget. The entropy-constrained portfolio optimization solution can be obtained by solving the KKT conditions, which result in 
$$
x_i = \frac{e^{-\frac{c_i}{b}}}{\sum_j^d e^{-\frac{c_j}{b}}},
$$
where $d = |c|$, which can be solved by bisection, resulting in an approximation of the globally optimal solution. Bisection accuracy threshold is set to $10^{-10}$.

We compare four methods. 
\emph{Local DFL} is a traditional SPO+ model trained on local data only, separately for each client (in some cases, we introduce heterogeneity in data availability). \emph{Fed-DFFL} is the pure version of the proposed approach, whereby a FedAvg algorithm based on SPO+ loss is trained collaboratively. In addition, we also consider two more implementations of DFFL, \emph{Interp-DFFL(SPO+)} and \emph{Interp-DFFL(MSE)}, that interpolate between the pure local and federated model, described in detail next. 

\subsection{Interp-DFFL: local-federated personalization for DFL}
\label{sec:interp-dffl}

The bounds in \cref{sec:fed-gain} suggest that local and federated models should not be viewed as an all-or-nothing choice. Clients close to the mixture population should benefit from the statistical advantage of federation, while clients that are outliers in either objective distribution or feasible set may be harmed by a purely federated model. At the same time, since only upper bounds are available, it may not be possible to precisely separate these cases for a particular client. We therefore consider a simple post-training personalization step that interpolates between each client's local and DFFL predictors, following, for example, \citet{deng2020adaptive}.

For client $j$, define the interpolated model as
\begin{equation}
    \hat{g}^{\mathrm{int}}_{j,\lambda}(x)
    :=
    (1-\lambda)\hat{g}^{\mathrm{loc}}_j(x)
    +
    \lambda \hat{g}^{\mathrm{fed}}(x),
    \qquad \lambda \in [0,1].
    \label{eq:interp-predictor}
\end{equation}
The client then selects
\begin{equation}
    \hat{\lambda}_j \in
    \underset{\lambda \in \Lambda}{\arg\min}
    \frac{1}{|V_j|}
    \sum_{(x,c) \in V_j}
    \ell^{S_j}_{\mathrm{SPO+}}
    \left(\hat{g}^{\mathrm{int}}_{j,\lambda}(x), c\right),
    \label{eq:interp-lambda}
\end{equation}
where $V_j$ is a client-specific validation set and $\Lambda \subseteq [0,1]$ is a finite grid containing $0$ and $1$. We refer to this procedure as Interp-DFFL(SPO+). For comparison, we also evaluate Interp-DFFL(MSE), which selects $\lambda_j$ using validation mean squared error (MSE) while keeping the same SPO+-trained local and federated base models.
For the experiments, we use an 80/20 validation split only for the interpolation methods, while the pure local and federated methods have access to the entire training set. We select $\Lambda$ as all values between 0 and 1 with a 0.05 step size.

Since $\Lambda$ contains the endpoints, the selected interpolated model is no worse than the local or federated model on the client validation objective. However, relative to the pure local and federated baselines, interpolation introduces an additional model-selection step, since each client must reserve data for validation, and the selected $\hat{\lambda}_j$ may be noisy when the validation set is small or unrepresentative. Interp-DFFL therefore offers a trade-off. Potentially worse individual performance, but better performance at the federated level, since each client can select their own mixture. Throughout our experiments, $\lambda = 1$ denotes a fully federated weighting, and $\lambda = 0$ a local weighting.

\subsection{Synthetic experiments: local, federated, and interpolated behavior}
\label{subsec:client-hetero}
We first evaluate DFFL on synthetically generated data. Heterogeneity is introduced along two axes: \emph{objective heterogeneity}, through client-specific perturbations of the feature-to-cost mapping $x \mapsto c$, and \emph{constraint heterogeneity}, through client-specific feasible-set parameters, namely budgets in the knapsack problem and entropy thresholds in the portfolio problem. We also vary statistical heterogeneity through sample-size imbalance. In the balanced setting, all clients have the same sample size. In the imbalanced setting, half of the clients have ten times fewer samples. The data-generating process is described in the online companion in detail.

\subsubsection{Baseline local-federated behavior}
\begin{figure}[ht]
    \centering
    \begin{subfigure}{0.48\textwidth}
        \centering
        \includegraphics[width=\linewidth]
        {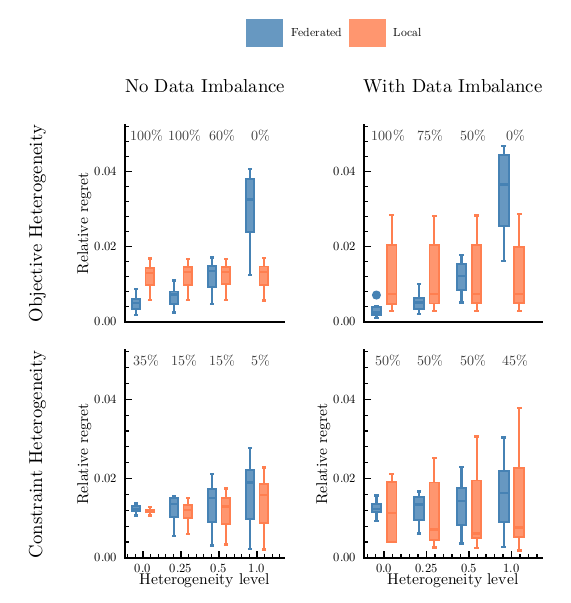}
        \caption{Federated versus local performance for the knapsack experiment.}
        \label{fig:knapsack-regret}
    \end{subfigure}
    \hfill
    \begin{subfigure}{0.48\textwidth}
        \centering
        \includegraphics[
            width=0.81\linewidth,
            trim=1.6cm 0 0 0, 
            clip
        ]
        {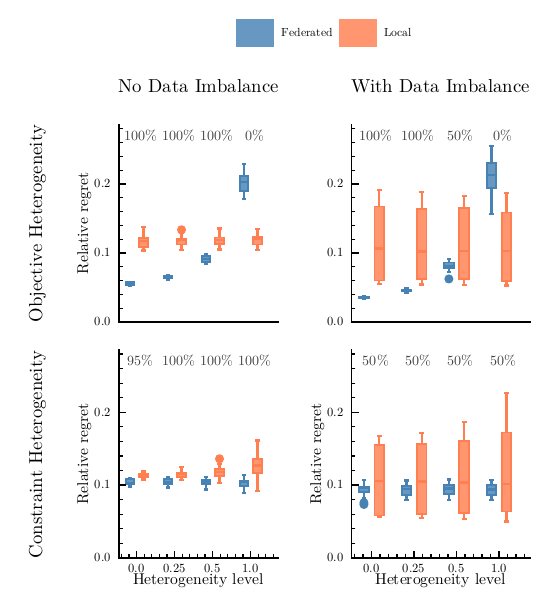}
        \caption{Federated versus local performance for the portfolio experiment.}
        \label{fig:portfolio-regret}
    \end{subfigure}
    \caption{Boxplots show the distribution of client-level relative regret for local and federated DFFL across heterogeneity levels and data imbalance for the two synthetic problems. Percentages indicate the fraction of clients for whom the federated model achieves lower regret, averaged across 5 seeds.}
    \label{fig:regret-combined}
\end{figure}

We first focus on the pure local and pure federated baselines in \cref{fig:regret-combined}. This isolates the empirical version of the tradeoff in Theorem 2 before considering interpolation.
\cref{fig:regret-combined} presents the mean regret trends across heterogeneity levels for all methods for federated (blue) and local (orange) cases.
In the fractional knapsack problem, whose feasible set is polyhedral, federation is sensitive to optimization heterogeneity. Objective heterogeneity is especially damaging: local models are largely unaffected because each client learns its own feature-to-cost map, while the federated model deteriorates as clients diverge. Constraint heterogeneity mainly increases cross-client variance, but it can further amplify the effect of objective heterogeneity. Under sample-size imbalance, data-poor clients benefit more from federation when heterogeneity is low or moderate, whereas data-rich clients switch toward local training sooner (this is particularly evident in the client-level plots presented in the Appendix.

The entropy-constrained portfolio problem in \cref{fig:regret-combined} behaves differently. Consistent with the strongly convex stability results in \cref{sec:strongly-convex}, federated models remain substantially more robust, especially under constraint heterogeneity. Federation degrades mainly under large objective heterogeneity, while constraint heterogeneity alone has a much smaller effect than in the polyhedral knapsack case. Overall, the synthetic baseline experiments support the central tradeoff presented in \cref{sec:fed-gain}: pooling helps statistically, but optimization heterogeneity can misalign the pooled model with individual clients.

\subsubsection{Interp-DFFL behavior}
\label{subsubsec:interp-dffl}

\begin{table}[t]
\centering
\small
\caption{Aggregate relative regret (\%) summary by experiment and method. Statistics are averaged over seeds and reported with inter-seed standard deviation. Mean, median, max, 75th and 90th percentiles are reported. "E" denotes the experiment, where K is the knapsack, and P is the portfolio. Lower is better $\downarrow$.}
\label{tab:aggregate-performance}
\begin{tabular}{lllllll}
\toprule
E & method & mean & median & p75 & p90 & max \\
\midrule
K & Local                & $1.21 \pm 0.15$ & $1.10 \pm 0.14$ & $1.73 \pm 0.20$ & $2.41 \pm 0.31$ & $6.34 \pm 1.45$ \\
K & Federated            & $1.37 \pm 0.18$ & $0.65 \pm 0.15$ & $1.83 \pm 0.31$ & $3.73 \pm 0.52$ & $11.70 \pm 0.83$ \\
K & Interp-DFFL(MSE)     & $1.10 \pm 0.14$ & $0.71 \pm 0.10$ & $1.37 \pm 0.16$ & $2.48 \pm 0.30$ & $10.36 \pm 2.04$ \\
K & Interp-DFFL(SPO+)    & $\mathbf{0.77 \pm 0.12}$ & $\mathbf{0.51 \pm 0.08}$ & $\mathbf{1.04 \pm 0.16}$ & $\mathbf{1.78 \pm 0.29}$ & $\mathbf{6.31 \pm 1.37}$ \\
P & Local                & $11.54 \pm 0.49$ & $11.29 \pm 0.23$ & $13.84 \pm 0.51$ & $17.19 \pm 1.54$ & $33.57 \pm 3.01$ \\
P & Federated            & $9.84 \pm 0.32$ & $7.02 \pm 0.37$ & $10.76 \pm 0.45$ & $22.00 \pm 0.96$ & $31.29 \pm 2.88$ \\
P & Interp-DFFL(MSE)     & $8.89 \pm 0.18$ & $7.65 \pm 0.32$ & $9.87 \pm 0.25$ & $15.72 \pm 0.61$ & $25.81 \pm 2.49$ \\
P & Interp-DFFL(SPO+)    & $\mathbf{7.84 \pm 0.26}$ & $\mathbf{6.86 \pm 0.31}$ & $\mathbf{9.46 \pm 0.27}$ & $\mathbf{12.69 \pm 0.55}$ & $\mathbf{21.78 \pm 3.39}$ \\
\bottomrule
\end{tabular}
\end{table}

In addition, we compare the interpolation methods with the pure baselines in \cref{fig:methods_overall_het,tab:aggregate-performance}. \cref{fig:methods_overall_het} shows that interpolation roughly matches the best of the two pure implementations, and even outperforms both federated and local models in the moderate heterogeneity regimes. \cref{tab:aggregate-performance} shows that Interp-DFFL(SPO+) improves aggregate performance in both synthetic domains. In knapsack, it achieves the lowest mean relative regret, $0.77\%$, compared with $1.21\%$ for local training and $1.37\%$ for federated training. In portfolio, it achieves $7.84\%$, compared with $11.54\%$ for local and $9.84\%$ for federated. The improvement is especially visible in the tail: in the portfolio case, the 90th percentile relative regret falls from $22.00\%$ under federation to $12.69\%$ under Interp-DFFL(SPO+). The comparison with Interp-DFFL(MSE) also supports the decision-focused validation criterion: selecting the interpolation weight by SPO+ validation loss yields lower mean and tail regret than selecting it by MSE, even though both variants use the same SPO+ trained base models.

\begin{figure}
    \centering
    \begin{minipage}{0.48\textwidth}
        \centering
        \includegraphics[width=\linewidth]{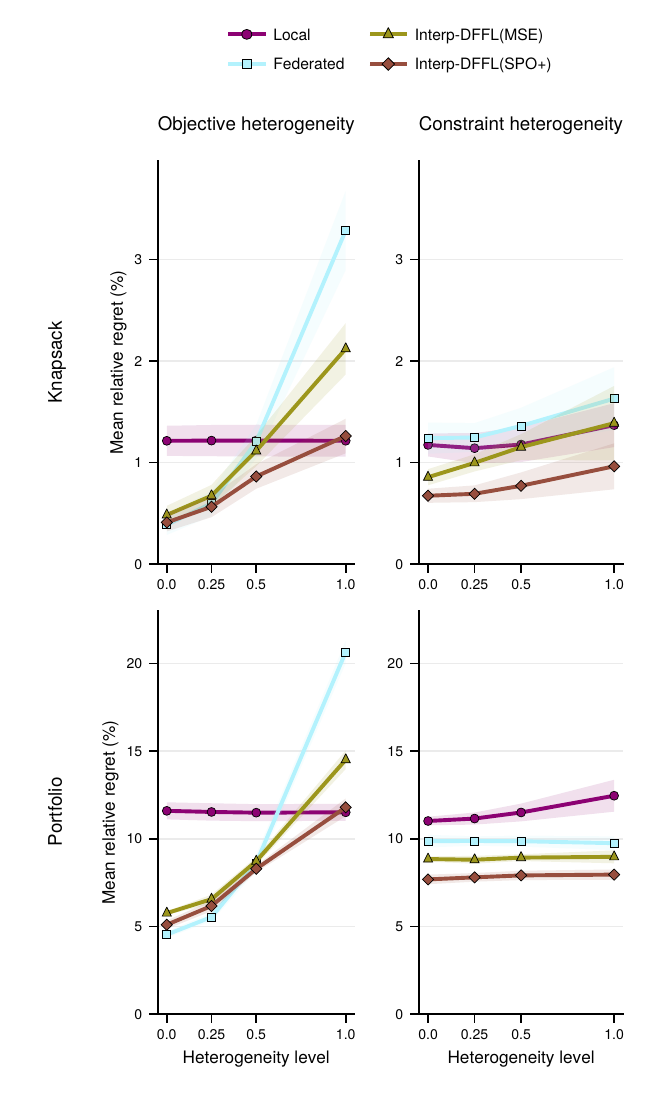}
        \caption{Mean relative regret by heterogeneity level for local, federated, and interpolated models. Bands denote $\pm$ one standard deviation across 5 seeds.}
        \label{fig:methods_overall_het}
    \end{minipage}
    \hfill
    \begin{minipage}{0.48\textwidth}
        \centering
        \includegraphics[width=\linewidth]{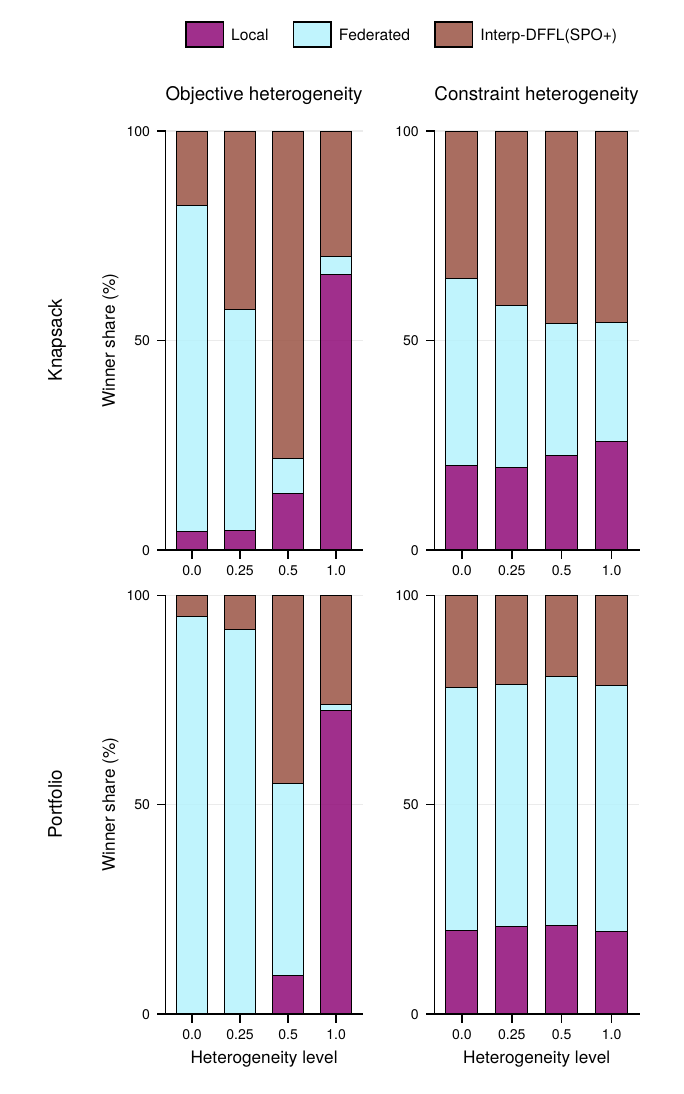}
        \caption{Client-level winner shares by heterogeneity level.}
        \label{fig:winner_share_het}
    \end{minipage}
\end{figure}

The winner-share results in \cref{fig:winner_share_het} provide a useful insight. Interp-DFFL(SPO+) is not the pointwise best method for all replications. In the portfolio experiments, the purely federated model is often the most frequent winner, which is consistent with the strong-convexity analysis: the federated model is usually stable and effective when the downstream feasible region is strongly convex. However, its failures are concentrated in higher heterogeneity configurations and can be large in magnitude. Interp-DFFL(SPO+) reduces these tail failures, which explains why it achieves better aggregate mean and tail regret despite not offering pointwise improvement.

Interp-DFFL(SPO+) also responds to statistical heterogeneity in the direction predicted by the theory. \cref{tab:interp-spo-summary} reports the selected interpolation weights under sample-size imbalance. Since larger $\lambda$ corresponds to more federated weight, data-poor clients rely more heavily on the federated model than data-rich clients: the mean $\lambda$ is $0.67$ versus $0.43$ in knapsack and $0.89$ versus $0.67$ in portfolio. \cref{fig:lambdas} shows the same pattern for optimization heterogeneity. In the polyhedral knapsack problem, clients move toward local models as either objective or constraint heterogeneity increases. In the strongly convex portfolio problem, the selected weights heavily skew towards the federated model under constraint heterogeneity and decrease mainly with objective heterogeneity.

Overall, we observe that even though the purely federated version of DFFL can be very sensitive to model heterogeneity, some form of federation is beneficial to most clients in most cases, even when heterogeneity is relatively high. 

\begin{figure}[t]
    \centering
    \begin{minipage}{0.42\textwidth}
        \centering
        \small
        \begin{tabular}{llcc}
        \toprule
         & Client Type & Mean R (\%) & Mean $\lambda$ \\
        \midrule
        K  & Data-rich & $0.38 \pm 0.04$ & $0.43 \pm 0.03$ \\
        K  & Data-poor  & $0.99 \pm 0.28$ & $0.67 \pm 0.07$ \\
        P & Data-rich & $5.51 \pm 0.15$ & $0.67 \pm 0.03$ \\
        P & Data-poor  & $8.43 \pm 0.39$ & $0.89 \pm 0.02$ \\
        \bottomrule
        \end{tabular}
        \captionof{table}{Interp-DFFL(SPO+) performance under the imbalanced training regime. Statistics are averaged over seeds and reported with inter-seed standard deviation.}
        \label{tab:interp-spo-summary}
    \end{minipage}
    \hfill
    \begin{minipage}{0.55\textwidth}
        \centering
        \includegraphics[width=\linewidth]{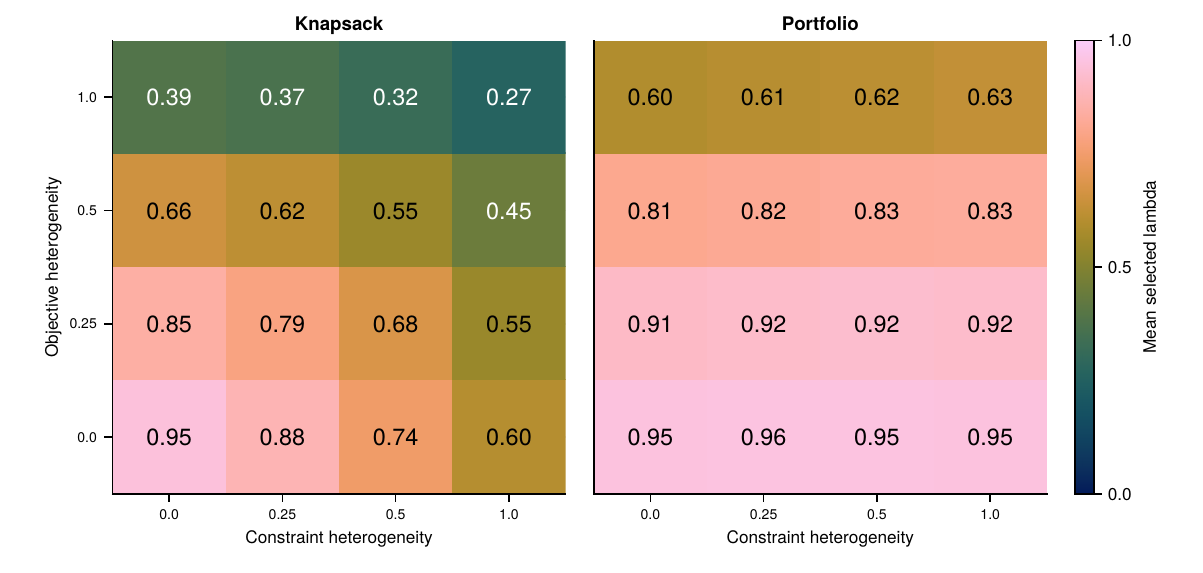}
        \captionof{figure}{Selected Interp-DFFL(SPO+) weights across objective and constraint heterogeneity levels.}
        \label{fig:lambdas}
    \end{minipage}
\end{figure}

\subsection{Case study: PJM Energy Markets}

Energy markets operations often result in difficult optimization and machine learning problems, due to the need to solve large-scale optimization problems under uncertainty, driven by demand variability and increasing renewable penetration \citep{bo2009probabilistic}. At the same time, operational data are often siloed across market participants and subject to strict regulatory requirements. For example, the North American Electric Reliability Corporation (NERC) Critical Infrastructure Protection standard (CIP 011-04) specifies protections for sensitive system information \citep{nerc0114cip}. These constraints can motivate distributed and privacy-aware optimization approaches, where techniques such as federated learning enable coordination across entities without requiring centralized data sharing.

For this set of experiments, we consider agents making decisions based on real-time electricity price, specifically employing hourly energy pricing data from the PJM Interconnection, a regional organization that manages the flow of electricity across 13 U.S. states \citep{pjm_lmp_2025}. We retrieve the hourly regional day-ahead energy prices for the year 2025. We then decompose the prices into a mean price and a per-client congestion term, which matches real-world pricing regimes \citep{litvinov2004marginal}.
Covariates are represented by the weather data, which has been employed in the literature as a strong predictor for locational marginal pricing for energy (LMP), which we retrieve through the Open-Meteo API \citep{hong2020locational, openmeteodata}. The covariates include minimum, maximum, and mean temperature, mean humidity, mean wind speed, mean and max solar irradiance, as well as the day of the week. To avoid time-series forecasting challenges, we do not use any historical pricing data as covariates. For each client $j$, we use six months of data for training and evaluate performance on the remainder of the year. We use a shallow neural net for prediction with 64 neurons. The training parameters are otherwise identical to the ones used for synthetic data.

We assume that each PJM zone has a decision-making agent that solves a daily energy purchasing problem. 
We arbitrarily assign agent demand levels to each zone, corresponding to 4 to 12 hour-long blocks. 
\cref{fig:pjm-case-study} depicts zone (agent) locations and corresponding labels and assigned demand levels. 
The downstream decision problem is to satisfy the daily demand by purchasing on the day-ahead market, i.e., each agent $j$ commits to selecting a subset of $B_j$ hour blocks in the next 24 hours based on available covariate context (weather forecast), such that the realized cost is minimal. Formally, let $S_j = \{w\in[0,1]^m:\mathbf{1}^Tw=B_j\}$ be the feasible set of client $j$, where $B_j$ is the client $j$ total daily demand and $m = 24$ hours. The client then solves,
$\min_{w\in S_j} c_j^Tw,
$
where $c_j$ is the vector of hourly energy prices. This defines a top-k selection problem, and the prediction model is used to identify the cheapest purchasing hours each day.
For example, it could represent determining a daily scheduling of appliances or EV charging. 
Importantly, note that as formulated, the agents significantly differ in their feasible regions ($B_j$) and objective coefficients (different prices at different nodes of the network).

\begin{figure}[t]
    \centering
    \includegraphics[
        width=0.9\textwidth,
        trim=0 0 0 6.5cm,
        clip
    ]{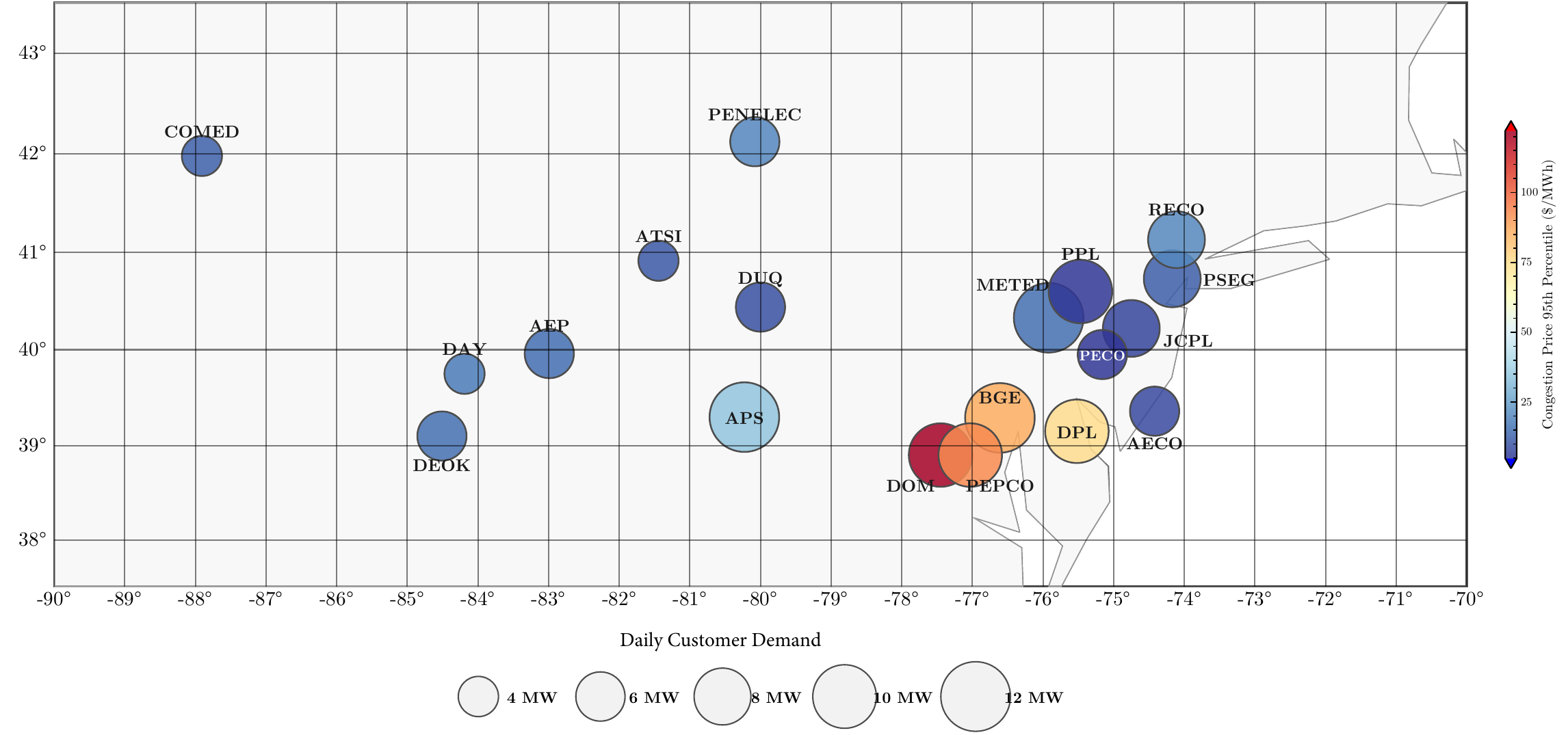}
    \caption{PJM client zones used in the case study. Node size indicates the assigned daily demand $B_j$, measured in hourly blocks, and node color indicates the 95th percentile of the client's congestion-price component.}
    \label{fig:pjm-case-study}
\end{figure}

\begin{table}[t]
\small
\centering
\begin{tabular}{lcccc}
\toprule
Method & Mean \% change & Median \% change & Worst harm & Worst-20\% harm \\
\midrule
Fed-DFFL 
& $-8.23 \pm 0.59$ 
& $-10.45 \pm 2.09$ 
& $+21.23 \pm 4.96$ 
& $+8.78 \pm 3.12$ \\
Interp-DFFL 
& $\mathbf{-9.98 \pm 0.82}$ 
& $\mathbf{-10.79 \pm 1.86}$ 
& $\mathbf{+5.83 \pm 4.20}$ 
& $\mathbf{+2.44 \pm 1.49}$ \\
\bottomrule
\end{tabular}
\caption{Change (\%) in regret relative to the local models across the two methods: Federated DFFL and Interpolated DFFL for the PJM experiments. Metrics are computed across clients and reported as mean $\pm$ standard deviation across 5 seeds. Worst-20\% harm denotes the mean one-sided harm among the worst 20\% of clients.}
\label{tab:pjm-results}
\end{table}

\cref{fig:pjm_results} compares client-level regret changes across the clients and shows that Interp-DFFL reduces the largest harms from federation while preserving most of its average benefit. \cref{fig:pjm_results} and \cref{tab:pjm-results} show that this case study follows the same pattern as the synthetic experiments. Fed-DFFL improves average performance relative to local training, with an average regret reduction of $8.23\%$. However, this average gain masks client-level downside risk. Some clients are harmed by the federated model, with worst-client harm reaching $21.23\%$ and worst-20\% harm $8.78\%$. This showcases the local--federated tradeoff in \cref{sec:fed-gain}, as the pooled model can improve statistical efficiency while still being misaligned with atypical clients. Interp-DFFL reduces this downside. By selecting a client-specific local-federated mixture using validation SPO+ loss, it improves the mean regret reduction from $8.23\%$ to $9.98\%$ and reduces worst-client harm from $21.23\%$ to $5.83\%$. The worst-20\% harm is reduced from $8.78\%$ to $2.44\%$. Thus, as in the synthetic experiments, interpolation does not win for every client, but protects against poor performance for a minority of clients.

 Client-level regrets and selected interpolation weights are reported in \cref{tab:pjm-detailed-tab}.
The selected interpolation weights provide additional evidence that the validation step is responding to client heterogeneity. Several zones select $\lambda=1$, effectively using the federated model, which indicates that for many clients the pooled model is well aligned with local decision risk, while 5 clients prefer local models almost exclusively. Without interpolation, these two groups may struggle to federate effectively, as a subset would always be harmed in the chosen regime. The rest benefit from federation at least somewhat. 

\begin{figure}
    \centering
    \includegraphics[width=0.9\linewidth]{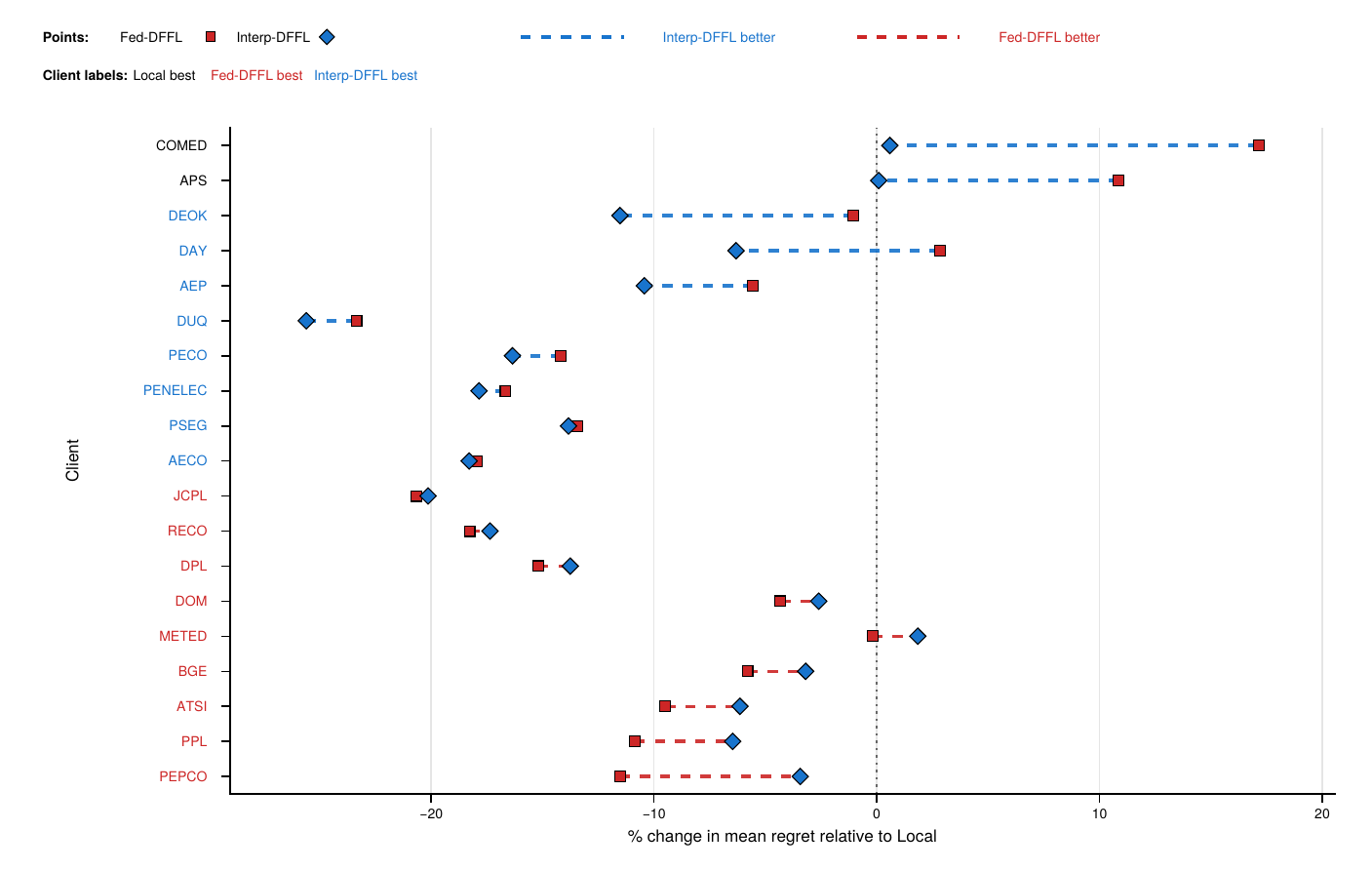}
    \caption{Client-level percentage change in regret relative to local DFFL on the PJM case study. Negative values indicate improvement over the local model. Positive values indicate harm.
    }
    \label{fig:pjm_results}
\end{figure}

\begin{table}[t]
\small
\centering
\begin{tabular}{@{}l|c|c|c|c@{}}
\hline
\textbf{Client} & $\bm{\mathrm{DFL}}_{\mathrm{LOC}}$ & $\bm{\mathrm{DFL}}_{\mathrm{FED}}$ & $\bm{\mathrm{DFL}}_{\mathrm{INTERP}}$ & $\bm{\lambda}$ \\
\hline
\textbf{DUQ}     & 21.55 $\pm$ 0.62 & 16.52 $\pm$ 0.79 & 16.03 $\pm$ 1.16 & 1.00 $\pm$ 0.00 \\
\textbf{JCPL}    & 46.77 $\pm$ 0.62 & 37.11 $\pm$ 0.46 & 37.35 $\pm$ 0.92 & 1.00 $\pm$ 0.00 \\
\textbf{RECO}    & 43.84 $\pm$ 0.69 & 35.83 $\pm$ 2.01 & 36.23 $\pm$ 1.41 & 1.00 $\pm$ 0.00 \\
\textbf{AECO}    & 35.58 $\pm$ 0.97 & 29.20 $\pm$ 1.33 & 29.07 $\pm$ 1.49 & 1.00 $\pm$ 0.00 \\
\textbf{PENELEC} & 35.14 $\pm$ 1.10 & 29.28 $\pm$ 0.69 & 28.86 $\pm$ 0.64 & 1.00 $\pm$ 0.00 \\
\textbf{DPL}     & 57.56 $\pm$ 0.99 & 48.82 $\pm$ 2.14 & 49.64 $\pm$ 1.42 & 1.00 $\pm$ 0.00 \\
\textbf{PECO}    & 36.48 $\pm$ 1.58 & 31.31 $\pm$ 1.56 & 30.52 $\pm$ 0.98 & 1.00 $\pm$ 0.00 \\
\textbf{PSEG}    & 44.90 $\pm$ 2.31 & 38.87 $\pm$ 1.04 & 38.69 $\pm$ 1.28 & 1.00 $\pm$ 0.00 \\
\textbf{PEPCO}   & 66.99 $\pm$ 2.41 & 59.27 $\pm$ 2.60 & 64.69 $\pm$ 1.93 & 0.00 $\pm$ 0.00 \\
\textbf{PPL}     & 46.99 $\pm$ 1.57 & 41.89 $\pm$ 1.19 & 43.95 $\pm$ 4.50 & 0.74 $\pm$ 0.38 \\
\textbf{ATSI}    & 13.84 $\pm$ 0.58 & 12.53 $\pm$ 0.43 & 12.99 $\pm$ 0.80 & 0.38 $\pm$ 0.28 \\
\textbf{BGE}     & 66.45 $\pm$ 1.82 & 62.60 $\pm$ 3.77 & 64.33 $\pm$ 1.66 & 0.00 $\pm$ 0.00 \\
\textbf{AEP}     & 22.63 $\pm$ 0.47 & 21.38 $\pm$ 0.80 & 20.27 $\pm$ 0.73 & 1.00 $\pm$ 0.00 \\
\textbf{DOM}     & 63.63 $\pm$ 1.72 & 60.87 $\pm$ 2.84 & 61.97 $\pm$ 1.27 & 0.00 $\pm$ 0.00 \\
\textbf{DEOK}    & 23.64 $\pm$ 0.34 & 23.39 $\pm$ 0.66 & 20.91 $\pm$ 1.46 & 0.88 $\pm$ 0.17 \\
\textbf{METED}   & 43.93 $\pm$ 1.20 & 43.85 $\pm$ 0.65 & 44.74 $\pm$ 0.75 & 0.01 $\pm$ 0.02 \\
\textbf{DAY}     & 13.54 $\pm$ 0.36 & 13.92 $\pm$ 0.78 & 12.68 $\pm$ 0.40 & 0.57 $\pm$ 0.40 \\
\textbf{APS}     & 48.40 $\pm$ 2.14 & 53.66 $\pm$ 4.94 & 48.45 $\pm$ 1.37 & 0.00 $\pm$ 0.00 \\
\textbf{COMED}   & 14.51 $\pm$ 0.33 & 17.00 $\pm$ 0.79 & 14.60 $\pm$ 0.49 & 0.54 $\pm$ 0.10 \\
\hline
\end{tabular}
\caption{Detailed results of the PJM case study.
$\bm{\mathrm{DFL}}_{\mathrm{LOC}}$: The downstream regret for each client when using SPO+ with a local-only model.
$\bm{\mathrm{DFL}}_{\mathrm{FED}}$: The downstream regret for each client when using SPO+ with a federated learning model.
$\bm{\mathrm{DFL}}_{\mathrm{INTERP}}$: The downstream regret for each client when using SPO+ with an interpolated local/federated model.
$\bm{\lambda}$: The interpolation weight selected per client (1 = federated, 0 = local).
All values are reported as mean $\pm$ standard deviation across 5 seeds.}
\label{tab:pjm-detailed-tab}
\end{table}

\section{Conclusions and Future Work}
\label{sec:conclusions}
In this paper, we introduced Decision-Focused Federated Learning under heterogeneous objectives and constraints. Using a support-function representation of SPO+, we derived heterogeneity bounds that separate objective shift from feasible-set shift. For general compact feasible sets, the bounds contain an unavoidable diameter-scale term, due to the discontinuity of optimizer mappings in polyhedral problems. Under strongly convex feasible sets, this term can be replaced by a stability-based control that vanishes with heterogeneity. We then lifted these pointwise discrepancy bounds to a local-versus-federated excess-risk comparison, and identified the central DFFL tradeoff: pooling data reduces statistical error, while optimization heterogeneity can misalign the pooled model with individual clients’ decision risks.

Our experiments support this tradeoff. In the polyhedral knapsack problem, federation is sensitive to objective and constraint heterogeneity, especially when clients have enough data to train effective local models. In the strongly convex portfolio problem, federation is more robust, especially under constraint heterogeneity. We also showed that a simple validation-based interpolation between local and federated DFFL models can further confirm the theory: clients choose more federated weight when data are scarce, or heterogeneity is low, and more local weight when optimization heterogeneity is high. In both synthetic experiments and the PJM energy-pricing case study, this interpolation reduces aggregate regret and client-level harm.

Several extensions remain open. First, the current theory compares local ERM and pooled ERM at the level of risks. Developing optimization and communication aware guarantees for specific DFFL algorithms remains an important direction. Second, the analysis assumes a shared covariate distribution in the client–mixture discrepancy bound. Extending the theory to covariate shift, for example through optimal transport or distributional discrepancy measures, would make the framework more realistic for cross-silo deployments. Third, the interpolation results suggest that personalization is valuable in DFFL, but the present approach is post-training and validation-based. Future work could incorporate objective and feasible-set heterogeneity directly into aggregation,  loss functions, personalized layers, or client-specific adaptation during training.

\section*{Acknowledgments}
This work was completed in part with resources provided by the Auburn University Easley Cluster.

\clearpage
\bibliographystyle{plainnat}
\bibliography{bib}

@book{boyd2004convex,
  title={Convex Optimization},
  author={Boyd, Stephen and Vandenberghe, Lieven},
  year={2004},
  publisher={Cambridge university press}
}

@article{elmachtoub2022smart,
  title={Smart “predict, then optimize”},
  author={Elmachtoub, Adam N and Grigas, Paul},
  journal={Management Science},
  volume={68},
  number={1},
  pages={9--26},
  year={2022},
  publisher={INFORMS}
}

@article{polovinkin1996strongly,
  title={Strongly convex analysis},
  author={Polovinkin, Evgenii Sergeevich},
  journal={Sbornik: Mathematics},
  volume={187},
  number={2},
  pages={259},
  year={1996},
  publisher={IOP Publishing}
}

@article{liu2021risk,
  title={Risk bounds and calibration for a smart predict-then-optimize method},
  author={Liu, Heyuan and Grigas, Paul},
  journal={Advances in Neural Information Processing Systems},
  volume={34},
  pages={22083--22094},
  year={2021}
}

@article{su2023non,
  title={A non-parametric view of fedavg and fedprox: Beyond stationary points},
  author={Su, Lili and Xu, Jiaming and Yang, Pengkun},
  journal={Journal of Machine Learning Research},
  volume={24},
  number={203},
  pages={1--48},
  year={2023}
}

@article{ho2022risk,
  title={Risk guarantees for end-to-end prediction and optimization processes},
  author={Ho-Nguyen, Nam and K{\i}l{\i}n{\c{c}}-Karzan, Fatma},
  journal={Management Science},
  volume={68},
  number={12},
  pages={8680--8698},
  year={2022},
  publisher={INFORMS}
}

@article{shalev2012online,
  title={Online Learning and Online Convex Optimization},
  author={Shalev-Shwartz, Shai},
  journal={Foundations and Trends{\textregistered} in Machine Learning},
  volume={4},
  number={2},
  pages={107--194},
  year={2012},
  publisher={Now Publishers Inc. Hanover, MA, USA}
}

@inproceedings{mcmahan2017communication,
  title={Communication-efficient learning of deep networks from decentralized data},
  author={McMahan, Brendan and Moore, Eider and Ramage, Daniel and Hampson, Seth and y Arcas, Blaise Aguera},
  booktitle={Artificial intelligence and statistics},
  pages={1273--1282},
  year={2017},
  organization={PMLR}
}

@article{el2019generalization,
  title={Generalization bounds in the predict-then-optimize framework},
  author={El Balghiti, Othman and Elmachtoub, Adam N and Grigas, Paul and Tewari, Ambuj},
  journal={Advances in Neural Information Processing Systems},
  volume={32},
  year={2019}
}

@book{schneider2013convex,
  title={Convex bodies: the Brunn--Minkowski theory},
  author={Schneider, Rolf},
  volume={151},
  year={2013},
  publisher={Cambridge University Press}
}

@article{bertsimas2020predictive,
  title={From predictive to prescriptive analytics},
  author={Bertsimas, Dimitris and Kallus, Nathan},
  journal={Management Science},
  volume={66},
  number={3},
  pages={1025--1044},
  year={2020},
  publisher={INFORMS}
}

@incollection{qi2022integrating,
  title={Integrating prediction/estimation and optimization with applications in operations management},
  author={Qi, Meng and Shen, Zuo-Jun},
  booktitle={Tutorials in Operations Research: Emerging and Impactful Topics in Operations},
  pages={36--58},
  year={2022},
  publisher={INFORMS}
}

@article{sadana2025survey,
  title={A survey of contextual optimization methods for decision-making under uncertainty},
  author={Sadana, Utsav and Chenreddy, Abhilash and Delage, Erick and Forel, Alexandre and Frejinger, Emma and Vidal, Thibaut},
  journal={European Journal of Operational Research},
  volume={320},
  number={2},
  pages={271--289},
  year={2025},
  publisher={Elsevier}
}

@article{kairouz2021advances,
  title={Advances and open problems in federated learning},
  author={Kairouz, Peter and McMahan, H Brendan and Avent, Brendan and Bellet, Aur{\'e}lien and Bennis, Mehdi and Bhagoji, Arjun Nitin and Bonawitz, Kallista and Charles, Zachary and Cormode, Graham and Cummings, Rachel and others},
  journal={Foundations and trends{\textregistered} in Machine Learning},
  volume={14},
  number={1--2},
  pages={1--210},
  year={2021},
  publisher={Now Publishers, Inc.}
}

@article{mandi2024decision,
  title={Decision-focused learning: Foundations, state of the art, benchmark and future opportunities},
  author={Mandi, Jayanta and Kotary, James and Berden, Senne and Mulamba, Maxime and Bucarey, Victor and Guns, Tias and Fioretto, Ferdinando},
  journal={Journal of Artificial Intelligence Research},
  volume={80},
  pages={1623--1701},
  year={2024}
}

@inproceedings{amos2017optnet,
  title={Optnet: Differentiable optimization as a layer in neural networks},
  author={Amos, Brandon and Kolter, J Zico},
  booktitle={International Conference on Machine Learning},
  pages={136--145},
  year={2017},
  organization={PMLR}
}

@inproceedings{khaled2020tighter,
  title={Tighter theory for local SGD on identical and heterogeneous data},
  author={Khaled, Ahmed and Mishchenko, Konstantin and Richt{\'a}rik, Peter},
  booktitle={International Conference on Artificial Intelligence and Statistics},
  pages={4519--4529},
  year={2020},
  organization={PMLR}
}

@inproceedings{
stich2018local,
title={Local {SGD} Converges Fast and Communicates Little},
author={Sebastian U. Stich},
booktitle={International Conference on Learning Representations},
year={2019},
url={https://openreview.net/forum?id=S1g2JnRcFX},
}

@inproceedings{
Li2020On,
title={On the Convergence of FedAvg on Non-IID Data},
author={Xiang Li and Kaixuan Huang and Wenhao Yang and Shusen Wang and Zhihua Zhang},
booktitle={International Conference on Learning Representations},
year={2020},
url={https://openreview.net/forum?id=HJxNAnVtDS}
}

@inproceedings{khodadadian2022federated,
  title={Federated reinforcement learning: Linear speedup under markovian sampling},
  author={Khodadadian, Sajad and Sharma, Pranay and Joshi, Gauri and Maguluri, Siva Theja},
  booktitle={International Conference on Machine Learning},
  pages={10997--11057},
  year={2022},
  organization={PMLR}
}

@inproceedings{flanagan2020federated,
  title={Federated multi-view matrix factorization for personalized recommendations},
  author={Flanagan, Adrian and Oyomno, Were and Grigorievskiy, Alexander and Tan, Kuan E and Khan, Suleiman A and Ammad-Ud-Din, Muhammad},
  booktitle={Joint European conference on machine learning and knowledge discovery in databases},
  pages={324--347},
  year={2020},
  organization={Springer}
}

@article{zhang2021federated,
  title={Federated learning with domain generalization},
  author={Zhang, Liling and Lei, Xinyu and Shi, Yichun and Huang, Hongyu and Chen, Chao},
  journal={arXiv preprint arXiv:2111.10487},
  year={2021}
}

@article{fallah2020personalized,
  title={Personalized federated learning with theoretical guarantees: A model-agnostic meta-learning approach},
  author={Fallah, Alireza and Mokhtari, Aryan and Ozdaglar, Asuman},
  journal={Advances in neural information processing systems},
  volume={33},
  pages={3557--3568},
  year={2020}
}

@inproceedings{li2021ditto,
  title={Ditto: Fair and robust federated learning through personalization},
  author={Li, Tian and Hu, Shengyuan and Beirami, Ahmad and Smith, Virginia},
  booktitle={International conference on machine learning},
  pages={6357--6368},
  year={2021},
  organization={PMLR}
}

@article{sattler2020clustered,
  title={Clustered federated learning: Model-agnostic distributed multitask optimization under privacy constraints},
  author={Sattler, Felix and M{\"u}ller, Klaus-Robert and Samek, Wojciech},
  journal={IEEE transactions on neural networks and learning systems},
  volume={32},
  number={8},
  pages={3710--3722},
  year={2020},
  publisher={IEEE}
}

@inproceedings{elmachtoub2020decision,
  title={Decision trees for decision-making under the predict-then-optimize framework},
  author={Elmachtoub, Adam N and Liang, Jason Cheuk Nam and McNellis, Ryan},
  booktitle={International conference on machine learning},
  pages={2858--2867},
  year={2020},
  organization={PMLR}
}

@inproceedings{mandi2020smart,
  title={Smart predict-and-optimize for hard combinatorial optimization problems},
  author={Mandi, Jayanta and Stuckey, Peter J and Guns, Tias and others},
  booktitle={Proceedings of the AAAI conference on artificial intelligence},
  volume={34},
  pages={1603--1610},
  year={2020}
}

@inproceedings{jeong2022exact,
  title={An exact symbolic reduction of linear smart predict+ optimize to mixed integer linear programming},
  author={Jeong, Jihwan and Jaggi, Parth and Butler, Andrew and Sanner, Scott},
  booktitle={International Conference on Machine Learning},
  pages={10053--10067},
  year={2022},
  organization={PMLR}
}

@inproceedings{augenstein2020generative,
  title={Generative Models for Effective ML on Private, Decentralized Datasets},
  author={Augenstein, Sean and McMahan, H Brendan and Ramage, Daniel and Ramaswamy, Swaroop and Kairouz, Peter and Chen, Mingqing and Mathews, Rajiv and y Arcas, Blaise Aguera},
  year={2020},
  booktitle={International Conference on Learning Representations}
}

@article{zhuo2019federated,
  title={Federated deep reinforcement learning},
  author={Zhuo, Hankz Hankui and Feng, Wenfeng and Lin, Yufeng and Xu, Qian and Yang, Qiang},
  journal={arXiv preprint arXiv:1901.08277},
  year={2019}
}

@article{li2020federated,
  title={Federated optimization in heterogeneous networks},
  author={Li, Tian and Sahu, Anit Kumar and Zaheer, Manzil and Sanjabi, Maziar and Talwalkar, Ameet and Smith, Virginia},
  journal={Proceedings of Machine learning and systems},
  volume={2},
  pages={429--450},
  year={2020}
}

@article{mivsic2020data,
  title={Data analytics in operations management: A review},
  author={Mi{\v{s}}i{\'c}, Velibor V and Perakis, Georgia},
  journal={Manufacturing \& Service Operations Management},
  volume={22},
  number={1},
  pages={158--169},
  year={2020},
  publisher={INFORMS}
}

@article{halkin1974implicit,
  title={Implicit functions and optimization problems without continuous differentiability of the data},
  author={Halkin, Hubert},
  journal={SIAM Journal on Control},
  volume={12},
  number={2},
  pages={229--236},
  year={1974},
  publisher={SIAM}
}

@inproceedings{domke2010implicit,
  title={Implicit differentiation by perturbation},
  author={Domke, Justin},
  booktitle={Proceedings of the 24th International Conference on Neural Information Processing Systems},
  pages={523--531},
  year={2010}
}

@inproceedings{poganvcic2019differentiation,
  title={Differentiation of blackbox combinatorial solvers},
  author={Pogan{\v{c}}i{\'c}, Marin Vlastelica and Paulus, Anselm and Musil, Vit and Martius, Georg and Rolinek, Michal},
  booktitle={International Conference on Learning Representations},
  year={2019}
}

@inproceedings{kotary2021end,
  title={End-to-End Constrained Optimization Learning: A Survey},
  author={Kotary, James and Fioretto, Ferdinando and Van Hentenryck, Pascal and Wilder, Bryan},
  booktitle={International Joint Conference on Artificial Intelligence},
  year={2021}
}

@article{smith2017federated,
  title={Federated multi-task learning},
  author={Smith, Virginia and Chiang, Chao-Kai and Sanjabi, Maziar and Talwalkar, Ameet S},
  journal={Advances in neural information processing systems},
  volume={30},
  year={2017}
}

@inproceedings{karimireddy2020scaffold,
  title={Scaffold: Stochastic controlled averaging for federated learning},
  author={Karimireddy, Sai Praneeth and Kale, Satyen and Mohri, Mehryar and Reddi, Sashank and Stich, Sebastian and Suresh, Ananda Theertha},
  booktitle={International conference on machine learning},
  pages={5132--5143},
  year={2020},
  organization={PMLR}
}

@article{wang2020tackling,
  title={Tackling the objective inconsistency problem in heterogeneous federated optimization},
  author={Wang, Jianyu and Liu, Qinghua and Liang, Hao and Joshi, Gauri and Poor, H Vincent},
  journal={Advances in neural information processing systems},
  volume={33},
  pages={7611--7623},
  year={2020}
}

@article{li2025centralized,
  title={From centralized to decentralized federated learning: Theoretical insights, privacy preservation, and robustness challenges},
  author={Li, Qiongxiu and Yu, Wenrui and Xia, Yufei and Pang, Jun},
  journal={arXiv preprint arXiv:2503.07505},
  year={2025}
}

@InProceedings{tang2024cave,
author="Tang, Bo
and Khalil, Elias B.",
editor="Dilkina, Bistra",
title="CaVE: A Cone-Aligned Approach for Fast Predict-then-optimize with Binary Linear Programs",
booktitle="Integration of Constraint Programming, Artificial Intelligence, and Operations Research",
year="2024",
publisher="Springer Nature Switzerland",
address="Cham",
pages="193--210",
isbn="978-3-031-60599-4"
}

@article{ye2023heterogeneous,
  title={Heterogeneous federated learning: State-of-the-art and research challenges},
  author={Ye, Mang and Fang, Xiuwen and Du, Bo and Yuen, Pong C and Tao, Dacheng},
  journal={ACM Computing Surveys},
  volume={56},
  number={3},
  pages={1--44},
  year={2023},
  publisher={ACM New York, NY, USA}
}

@article{gao2022survey,
  title={A survey on heterogeneous federated learning},
  author={Gao, Dashan and Yao, Xin and Yang, Qiang},
  journal={arXiv preprint arXiv:2210.04505},
  year={2022}
}

@article{bezanson2017julia,
  title={Julia: A fresh approach to numerical computing},
  author={Bezanson, Jeff and Edelman, Alan and Karpinski, Stefan and Shah, Viral B},
  journal={SIAM Review},
  volume={59},
  number={1},
  pages={65--98},
  year={2017},
  publisher={SIAM}
}

@article{pal2023lux,
  title={Lux: Explicit parameterization of deep neural networks in julia},
  author={Pal, Avik},
  journal={Zenodo},
  year={2023}
}

@article{innes2019differentiable,
  title={A differentiable programming system to bridge machine learning and scientific computing},
  author={Innes, Mike and Edelman, Alan and Fischer, Keno and Rackauckas, Chris and Saba, Elliot and Shah, Viral B and Tebbutt, Will},
  journal={arXiv preprint arXiv:1907.07587},
  year={2019}
}

@misc{pjm_lmp_2025,
  author       = {{PJM Interconnection, L.L.C.}},
  title        = {Real-Time Hourly Locational Marginal Pricing (LMP) Data},
  year         = {2025},
  howpublished = {\url{https://www.energyonline.com/}},
  note         = {Accessed via EnergyOnline. Data from Jan 1--Dec 31, 2025. Accessed: 2026-02-14}
}

@misc{openmeteodata,
  author = {Zippenfenig, Patrick},
  doi = {10.5281/zenodo.7970649},
  licence = {CC-BY-4.0},
  title = {Open-Meteo.com Weather API},
  year = {2023},
  copyright = {Creative Commons Attribution 4.0 International},
  url = {https://open-meteo.com/}
}

@article{hong2020locational,
  title={Locational marginal price forecasting in a day-ahead power market using spatiotemporal deep learning network},
  author={Hong, Ying-Yi and Taylar, Jonathan V and Fajardo, Arnel C},
  journal={Sustainable Energy, Grids and Networks},
  volume={24},
  pages={100406},
  year={2020},
  publisher={Elsevier}
}

@article{litvinov2004marginal,
  title={Marginal loss modeling in LMP calculation},
  author={Litvinov, Eugene and Zheng, Tongxin and Rosenwald, Gary and Shamsollahi, Payman},
  journal={IEEE transactions on Power Systems},
  volume={19},
  number={2},
  pages={880--888},
  year={2004},
  publisher={IEEE}
}

@article{bo2009probabilistic,
  title={Probabilistic LMP forecasting considering load uncertainty},
  author={Bo, Rui and Li, Fangxing},
  journal={IEEE Transactions on Power Systems},
  volume={24},
  number={3},
  pages={1279--1289},
  year={2009},
  publisher={IEEE}
}

@misc{nerc0114cip,
  author       = {{North American Electric Reliability Corporation}},
  title        = {{CIP-011-4 -- Cyber Security -- Information Protection}},
  year         = {2019},
  url     = {https://www.nerc.com/pa/Stand/Reliability%20Standards/CIP-011-4.pdf},
  urldate = {2026-02-14},
}

@article{ban2019big,
  title={The big data newsvendor: Practical insights from machine learning},
  author={Ban, Gah-Yi and Rudin, Cynthia},
  journal={Operations Research},
  volume={67},
  number={1},
  pages={90--108},
  year={2019},
  publisher={INFORMS}
}

@article{bertsimas2022data,
  title={Data-driven optimization: A reproducing kernel hilbert space approach},
  author={Bertsimas, Dimitris and Koduri, Nihal},
  journal={Operations Research},
  volume={70},
  number={1},
  pages={454--471},
  year={2022},
  publisher={INFORMS}
}

@article{kannan2025data,
  title={Data-driven sample average approximation with covariate information},
  author={Kannan, Rohit and Bayraksan, G{\"u}zin and Luedtke, James R},
  journal={Operations Research},
  volume={73},
  number={6},
  pages={3245--3259},
  year={2025},
  publisher={INFORMS}
}

@article{qi2025integrated,
  title={Integrated conditional estimation-optimization},
  author={Qi, Meng and Grigas, Paul and Shen, Zuo-Jun},
  journal={Operations Research},
  year={2025},
  publisher={INFORMS}
}

@article{chen2026robust,
  title={Robust actionable prescriptive analytics},
  author={Chen, Li and Sim, Melvyn and Zhang, Xun and Zhao, Long and Zhou, Minglong},
  journal={Operations Research},
  volume={74},
  number={1},
  pages={550--571},
  year={2026},
  publisher={INFORMS}
}

@article{chen2025algorithmic,
  title={An algorithmic approach to managing supply chain data security: The differentially private newsvendor},
  author={Chen, Du and Chua, Geoffrey A},
  journal={Operations Research},
  year={2025},
  publisher={INFORMS}
}

@article{ayer2019prioritizing,
  title={Prioritizing hepatitis C treatment in US prisons},
  author={Ayer, Turgay and Zhang, Can and Bonifonte, Anthony and Spaulding, Anne C and Chhatwal, Jagpreet},
  journal={Operations Research},
  volume={67},
  number={3},
  pages={853--873},
  year={2019},
  publisher={INFORMS}
}

@book{shalev2014understanding,
  title={Understanding machine learning: From theory to algorithms},
  author={Shalev-Shwartz, Shai and Ben-David, Shai},
  year={2014},
  publisher={Cambridge university press}
}

@article{m2024personalized,
  title={Personalized federated learning with mixture of models for adaptive prediction and model fine-tuning},
  author={M Ghari, Pouya and Shen, Yanning},
  journal={Advances in Neural Information Processing Systems},
  volume={37},
  pages={92155--92183},
  year={2024}
}

@article{deng2020adaptive,
  title={Adaptive personalized federated learning},
  author={Deng, Yuyang and Kamani, Mohammad Mahdi and Mahdavi, Mehrdad},
  journal={arXiv preprint arXiv:2003.13461},
  year={2020}
}

@article{mansour2020three,
  title={Three approaches for personalization with applications to federated learning},
  author={Mansour, Yishay and Mohri, Mehryar and Ro, Jae and Suresh, Ananda Theertha},
  journal={arXiv preprint arXiv:2002.10619},
  year={2020}
}
\appendix
\section{Implementation Details}
\label{appsec:experiments}

\subsection{Synthetic Data-Generation Process}
\label{appsubsec:synthetic_dgp}

Each synthetic instance is a pair $(x_i, c_i)$ with features $x_i \in \mathbb{R}^p$ and product-level costs $c_i \in \mathbb{R}^d$, with defaults $p=8$ and $d=50$. For each run, the random seed is fixed.

\paragraph{Feature generation.}
Features are sampled i.i.d. as
$
x_i \sim \mathcal{N}(0, I_p), \quad i=1,\dots,n.
$
To couple features to product costs, we sample a binary loading matrix $B \in \{0,1\}^{d \times p}$ once per dataset, with i.i.d.\ entries $B_{k\ell} \sim \mathrm{Bernoulli}(0.5)$, similar to the generation method used by \cite{liu2021risk}.

\paragraph{Cost generation.}
For product $k$ and sample $i$ we use a generating polynomial, 
\[
c_{ki} = \Bigl[1 + \Bigl(1 + \frac{B_k^\top x_i}{\sqrt{p}}\Bigr)^{\kappa}\Bigr] \xi_{ki},
\]
where $\kappa$ controls nonlinearity and $\xi_{ki}$ is multiplicative noise:
\[
\xi_{ki} =
\begin{cases}
1, & \epsilon = 0,\\
\mathrm{Uniform}(1-\epsilon, 1+\epsilon), & \epsilon > 0.
\end{cases}
\]

\paragraph{Client assignment and data imbalance.}
Samples are randomly permuted and assigned to 20 clients. Two configurations are used to test heterogeneity through client sample sizes, as in \citep{su2023non}:
\begin{enumerate}
    \item \textbf{Balanced Data}: $n=2000$, with equal allocation (100 samples/client).
    \item \textbf{Data-scarcity}: $n=5500$, with strong imbalance: clients 1--10 receive 500 samples each and clients 11--20 receive 50 samples each.
\end{enumerate}

\paragraph{Client-specific parameters (heterogeneity controls).}
Beyond predictive data $(x_i,c_i)$, each client receives optimization parameters controlling constraint heterogeneity, while objective heterogeneity is introduced through the predictive map. Specifically, if sample $i$ is assigned to client $j$, the generating polynomial uses a client-specific loading matrix
\[
B^{(j)} = B R_j,
\]
where $R_j$ is an orthogonal rotation generated from a skew-symmetric Gaussian matrix and normalized so that $|\eta_{\mathrm{obj}}|$ controls the rotation magnitude. When $\eta_{\mathrm{obj}} = 0$, $R_j = I_p$.

\textbf{Portfolio Optimization.}
Constraint offsets are generated as
\[
h_j = \mathrm{clip}\!\left(\frac{\log d}{2} + \eta_{\mathrm{constr}} \nu_j,\ 10^{-6},\ \log d - 10^{-6}\right),
\]
where $\nu_j \sim \mathrm{Unif}(-1,1)$,
and the stored scalar is $r_j = -h_j$.

\textbf{Fractional Knapsack.}
Shared base parameters (knapsack weights $\alpha$ and their total bound $B$) are sampled once,
\[
a^0 \in [0.5,1.5]^d,\qquad B^0 = 0.6d.
\]
Client-level parameters are then,
\[
B_j = B^0 \exp(\eta_{\mathrm{constr}}\zeta_j), \quad \zeta_j \sim \mathcal{N}(0,1).
\]
Thus, $\eta_{\mathrm{obj}}$ scales heterogeneity in the mapping of covariates to cost vectors as before, and $\eta_{\mathrm{constr}}$ scales feasible set heterogeneity.

We conduct the experiment for:
\[
\eta_{\mathrm{obj}},\eta_{\mathrm{constr}} \in \{0,0.25,0.5,1.0\},
\]
across 5 seeds, degrees $\kappa \in \{2,4,6,8\}$, and noise levels $\epsilon \in \{0,1\}$.

Evaluation uses an independent synthetic test set of 20,000 samples from the same data-generation law (the same seed, degree, noise, and $\eta_{\mathrm{obj}}$ controls) with balanced random assignment over 20 clients.

The experiments were run on a Slurm cluster using a single CPU-only compute node. Each run allocated 40 CPU cores and 60 GB memory and executed for 20 hours wall-clock time, corresponding to 800 CPU core-hours (20 node-hours).

\subsection{Algorithm Modifications}
\label{appsubsec:algorithms}
We adapt the \texttt{FedAvg} algorithm by \citet{mcmahan2017communication} to accomodate DFL oracle solvers as a subroutine and the SPO+ gradient update. Specific changes made are shown in blue in \cref{alg:fedavg-dffl,alg:clientupdate-dffl}.

\begin{algorithm}[t]
\caption{FedAvg-DFFL}
\label{alg:fedavg-dffl}
\begin{algorithmic}[1]
\Require Total clients $K$; communication rounds $T$; client fraction $C$; local epochs $E$; learning rate $\eta$; local optimizer $\mathrm{Opt}$.
\State Initialize global parameters $w^{0}$
\For{$t = 0,1,\dots,T-1$}
  \State $m \gets \max(\lfloor CK \rfloor, 1)$ \Same
  \State $S_t \gets$ sample $m$ clients uniformly without replacement \Same
  \ForAll{$k \in S_t$ \textbf{in parallel}}
    \State $w_k^{t+1} \gets \textsc{ClientUpdateDFFL}(k, w^t, E, \eta, \mathrm{Opt})$ \Diff
    \State $n_k \gets |\mathcal{D}_k|$ \Same
  \EndFor
  \State $w^{t+1} \gets \sum_{k\in S_t}
  \frac{n_k}{\sum_{j\in S_t} n_j}\, w_k^{t+1}$ \Same
\EndFor
\State \Return $w^{F} \gets w^{T}$
\end{algorithmic}
\end{algorithm}

\begin{algorithm}[t]
\caption{\textsc{ClientUpdateDFFL} (run on client $k$)}
\label{alg:clientupdate-dffl}
\begin{algorithmic}[1]
\Require Client index $k$; initialization $w_{\mathrm{init}}$; client dataset $\mathcal{D}_k$; epochs $E$; learning rate $\eta$; optimizer $\mathrm{Opt}$; client-specific optimization parameters.
\State $w_k \gets w_{\mathrm{init}}$ \Same
\State Initialize local optimizer state \Same
\For{$i = 1,2,\dots,E$} \Same
  \For{each minibatch $b \subset \mathcal{D}_k$}
    \State $\hat{y} \gets f_{w_k}(b)$ \Same
    \State $g_{\hat{y}} \gets \nabla_{\hat{y}}\,
    \mathcal{L}_{\mathrm{SPO+}}(\hat{y}; b,\text{client params})$ \Diff
    \State $g_{w} \gets \nabla_{w_k} f_{w_k}(b)^{\top} g_{\hat{y}}$ \Diff
    \State $w_k \gets \mathrm{OptStep}(w_k, g_w, \eta)$ \Diff
  \EndFor
\EndFor
\State \Return $w_k$
\end{algorithmic}
\end{algorithm}

\begin{algorithm}[t]
\caption{Interp-DFFL Personalization}
\label{alg:interp-dffl}
\begin{algorithmic}[1]
\Require Federated model $w^{F}$; local DFFL models $\{w_k^{\mathrm{loc}}\}_{k=1}^{K}$; validation sets $\{\mathcal{V}_k\}_{k=1}^{K}$; interpolation grid $\Lambda \subset [0,1]$; validation loss $\ell_{\mathrm{val}} \in \{\ell_{\mathrm{SPO+}}, \mathrm{MSE}\}$.
\For{$k = 1,2,\dots,K$ \textbf{in parallel}}
  \ForAll{$\lambda \in \Lambda$}
    \State Define $h_{k,\lambda}(x) \gets (1-\lambda) f_{w_k^{\mathrm{loc}}}(x) + \lambda f_{w^{F}}(x)$ \Diff
  \EndFor
  \State $\lambda_k^\star \gets
  \arg\min_{\lambda \in \Lambda}
  \frac{1}{|\mathcal{V}_k|}
  \sum_{b \in \mathcal{V}_k}
  \ell_{\mathrm{val}}\!\left(h_{k,\lambda}; b,\text{client params}\right)$ \Diff
  \State Define $h_k^{\mathrm{Interp}}(x) \gets (1-\lambda_k^\star) f_{w_k^{\mathrm{loc}}}(x) + \lambda_k^\star f_{w^{F}}(x)$ \Diff
\EndFor
\State \Return $\{h_k^{\mathrm{Interp}}, \lambda_k^\star\}_{k=1}^{K}$
\end{algorithmic}
\end{algorithm}

\section{Proofs}
\label{appsec:proofs}
\begin{proof}{Proof of \cref{lem:support-set}.}
See \citep{schneider2013convex}. 
Fix $u$. For any $w\in S_1$ there exists $\tilde w\in S_2$ with $\|w-\tilde w\|\le \delta$.
Then $u^\top w \le u^\top \tilde w + \|u\| \|w-\tilde w\|
\le \xi_{S_2}(u)+\|u\| \delta$.
Taking $\max_{w\in S_1}$ gives $\xi_{S_1}(u)\le \xi_{S_2}(u)+\|u\| \delta$.
Repeat with $(S_1,S_2)$ swapped.
\end{proof}

\begin{proof}{Proof of \cref{lem:zstar-set}.}
Note $z_S^\star(c)= -\xi_S(-c)$ and apply \cref{lem:support-set}.
\end{proof}

\begin{proof}{Proof of \cref{lem:translation-invariance}.}
Let $S' = S + v, v \in \mathbb{R}^d$, the translated set. The SPO+ loss is made up of three terms, which we can observe under set translation:
\begin{itemize}
    \item Optimizer shifts: $w^{\star}_{S'} = w^{\star}_S + v$ \citep{boyd2004convex}.
    \item  Optimal value shifts: $z^{\star}_{S'}(c) = z^{\star}_S(c) + c^Tv.$
    \item The max term shifts similarly: $\max_{w \in S'} (c - 2\hat{c})^Tw = \max_{w \in S} (c - 2\hat{c})^Tw + (c - 2\hat{c})^Tv.$
\end{itemize}
Substituting the terms into $\ell_{S+v}(\hat c,c)$ and collecting the "shift" terms gives us the claim.
\end{proof}

\begin{proof}{Proof of \cref{lem:diameter-joint}.}
Write
\[
\xi_S(c_1-2\hat c)-\xi_S(c_2-2\hat c)
\le \max_{w\in S}(c_1-c_2)^\top w,
\]
and
\[
-z_S^\star(c_1)+z_S^\star(c_2)
\le \max_{v\in S}(c_2-c_1)^\top v.
\]
Adding and taking maxima yields
$$
\Big(\xi_S(c_1-2\hat c)-\xi_S(c_2-2\hat c)\Big)
+
\Big(-z_S^\star(c_1)+z_S^\star(c_2)\Big)
\\ \le \max_{w,v\in S}(c_1-c_2)^\top(w-v)
\le \|c_1-c_2\|\,D(S).
$$
Swap $(c_1,c_2)$ to get the absolute value bound.
\end{proof}

\begin{proof}{Proof of \cref{lem:oracle-free}.}

Let $\delta=d_H(S_1,S_2)$. Choose $\tilde w_2\in S_2$ such that $\|w_1-\tilde w_2\|\le \delta$.
Then
\[
\begin{aligned}
|\hat c^\top w_1 - \hat c^\top w_2|
&\le |\hat c^\top(w_1-\tilde w_2)| + |\hat c^\top(\tilde w_2-w_2)| \\
&\le \|\hat c\| \|w_1-\tilde w_2\|
  + \|\hat c\| \|\tilde w_2-w_2\| \\
&\le \|\hat c\| (\delta + D(S_2)).
\end{aligned}
\]

By symmetry we also have the bound $\|\hat c\| (\delta + D(S_1))$, taking the minimum completes the bound.
\end{proof}

\begin{proof}{Proof of \cref{thm:main}}
Using \eqref{eq:spo+},
\[
\ell_{S}(\hat c,c)=\big(\xi_S(c - 2\hat{c})-z_S^\star(c)\big)+2\hat c^\top w_S^\star(c).
\]
Hence,
\[
\bigl|\ell_{S_{\rm ref}}(\hat c,c_{\rm ref}) - \ell_{S_{\rm cl}}(\hat c,c_{\rm cl})\bigr|
\le T + B,
\]
where
\[
\begin{aligned}
T &:= \Big|
\big(\xi_{S_{\rm ref}}(c_{\rm ref} - 2\hat{c})-z_{S_{\rm ref}}^\star(c_{\rm ref})\big) \\
&\quad -\big(\xi_{S_{\rm cl}}(c_{\rm cl} - 2\hat{c})-z_{S_{\rm cl}}^\star(c_{\rm cl})\big)
\Big|,\\
B &:= 2\bigl|\hat c^\top(w^\star_{S_{\rm ref}}(c_{\rm ref}) - w^\star_{S_{\rm cl}}(c_{\rm cl}))\bigr|.
\end{aligned}
\]

\textbf{Step 1 (bound $T$ by diameter + Hausdorff).}
Add and subtract $\xi_{S_{\rm ref}}(c_{\rm cl}-2\hat c)-z_{S_{\rm ref}}^\star(c_{\rm cl})$:
\begin{dmath}
T \le 
\Big|
\big(\xi_{S_{\rm ref}}(c_{\rm ref} - 2\hat{c})-z_{S_{\rm ref}}^\star(c_{\rm ref})\big)
-
\big(\xi_{S_{\rm ref}}(c_{\rm cl} - 2\hat{c})-z_{S_{\rm ref}}^\star(c_{\rm cl})\big)
\Big|  
+
\Big|
\big(\xi_{S_{\rm ref}}(c_{\rm cl} - 2\hat{c})-z_{S_{\rm ref}}^\star(c_{\rm cl})\big)
-
\big(\xi_{S_{\rm cl}}(c_{\rm cl} - 2\hat{c})-z_{S_{\rm cl}}^\star(c_{\rm cl})\big)
\Big|.
\end{dmath}
The first term is bounded by \cref{lem:diameter-joint}:
\[
\le D_{\rm ref}\,\|c_{\rm ref}-c_{\rm cl}\| = D_{\rm ref} c_d.
\]
For the second term, use \cref{lem:support-set} on the support term, \cref{lem:zstar-set} on the value term, and \cref{lem:translation-invariance} on the Hausdorff distance $\delta$:
\[
\le \delta_N\|c_{\rm cl} - 2\hat{c}\| + \delta_N\|c_{\rm cl}\|.
\]
So
\[
T \le D_{\rm ref}\,c_d + \delta_N\big(\|c_{\rm cl} - 2\hat{c}\|+\|c_{\rm cl}\|\big).
\]
Swapping the roles of ref/cl yields
\[
T \le D_{\rm cl}\,c_d + \delta_N\big(\|c_{\rm ref} - 2\hat{c}\|+\|c_{\rm ref}\|\big),
\]
so we may take the minimum of the two bounds.

\textbf{Step 2 (bound $B$).}
Because $w^\star_{S_{\rm ref}}(c_{\rm ref})\in S_{\rm ref}$ and $w^\star_{S_{\rm cl}}(c_{\rm cl})\in S_{\rm cl}$,
\cref{lem:oracle-free} yields
\[
B \le 2\|\hat c\| (\delta_N + D_{\min}).
\]

Combining the bounds on $T$ and $B$ gives the claim.
\end{proof}
\begin{proof}{Proof of \cref{lem:dir-stability}}
Under $\rho$-strong convexity (\citep{polovinkin1996strongly}, Corollary 4), the support function
$\xi_S(\cdot)$ is differentiable on $\mathbb{R}^{d}\setminus\{0\}$ and its gradient is $\rho$-Lipschitz
on the unit sphere. Moreover, $\nabla \xi_S(p) = x_S(p)$. This yields the stated inequality.
The second claim follows from $w_S^\star(c)=x_S(c/\|c\|)$ for $c\neq 0$.
\end{proof}

\begin{proof}{Proof of \cref{lem:QG}}
(\citep{polovinkin1996strongly}, Corollary 4) This is a standard consequence of the supporting-ball characterization of $\rho$-strong convexity:
the supporting ball at direction $p$ yields the inequality
$\|x-x_S(p)\|^2 + 2\rho (x-x_S(p))^\top p \le 0$ for all $x\in S$.
\end{proof}

\begin{proof}{Proof of \cref{lem:set-stability}}
Fix $p \in \mathbb{R}^{d}, \|p\|=1,$ and let $x_1:=x_{S_1}(p)$, $x_2:=x_{S_2}(p)$, where $x_S := \arg \max_{x \in S} p^Tx$, and $p^Tx_S = \xi_S(p)$.
By the definition of Hausdorff distance, there exists $y\in S_2$ with
$\|y-x_1 \| \le \delta$, hence $p^\top y \ge p^\top x_1 - \delta$ by Cauchy-Schwarz inequality.

By \cref{lem:support-set} we have
$|\xi_{S_2}(p) - \xi_{S_1}(p)| \le \delta$, so $\xi_{S_2}(p) \le p^\top x_1 + \delta$.
Combining yields $\xi_{S_2}(p) - p^\top y \le 2\delta$.

Applying \cref{lem:QG} to $S_2$ gives
$\frac{1}{2\rho_2}\|y-x_2 \|^2 \le 2\delta$, hence $\|y-x_2 \| \le 2\sqrt{\rho_2\delta}$.
Therefore,
\[
\|x_1-x_2 \| \le \|x_1-y \| + \|y-x_2 \| \le \delta + 2\sqrt{\rho_2\delta}.
\]
Swapping the roles of $S_1,S_2$ yields $\|x_1-x_2 \| \le \delta + 2\sqrt{\rho_1\delta}$.
Taking the minimum provides the claim.
\end{proof}

\begin{proof}{Proof of \cref{lem:lemma4-sc}}
By the triangle inequality,
$$
\|w^\star_{S_{\rm ref}}(c_{\rm ref}) - w^\star_{S_{\rm cl}}(c_{\rm cl}) \|
\le
\|w^\star_{S_{\rm ref}}(c_{\rm ref}) - w^\star_{S_{\rm ref}}(c_{\rm cl}) \|
+
\|w^\star_{S_{\rm ref}}(c_{\rm cl}) - w^\star_{S_{\rm cl}}(c_{\rm cl}) \|.
$$
The first term is bounded by \cref{lem:dir-stability} (directional stability on $S_{\rm ref}$),
giving $\le \rho_{\rm ref}\|p(c_{\rm ref})-p(c_{\rm cl}) \|$.
The second term is bounded by \cref{lem:set-stability} at direction $p(c_{\rm cl})$,
giving $\le \delta + 2\sqrt{\rho_{\rm min}\delta}$. Therefore, the RHS yields $\leq \rho_{\rm ref} \|p(c_{\rm ref}) - p(c_{\rm cl}) \| + \delta +2\sqrt{\rho_{\min}\delta}.$

Similarly, by triangle inequality,
$$
\|w^\star_{S_{\rm ref}}(c_{\rm ref}) - w^\star_{S_{\rm cl}}(c_{\rm cl}) \|
\le
\|w^\star_{S_{\rm ref}}(c_{\rm ref}) - w^\star_{S_{\rm cl}}(c_{\rm ref}) \|
+
\|w^\star_{S_{\rm cl}}(c_{\rm ref}) - w^\star_{S_{\rm cl}}(c_{\rm cl}) \|,
$$
where here the first term is bounded by \cref{lem:set-stability} and the second by \cref{lem:dir-stability}, and the RHS yields $\leq \rho_{\rm cl} \|p(c_{\rm ref}) - p(c_{\rm cl}) \| + \delta +2\sqrt{\rho_{\min}\delta}.$

Taking the minimum of the two resulting bounds, since $\delta + 2\sqrt{\rho_{\min} \delta}$ is the same across both inequalities, gives us,

$$
\|w^\star_{S_{\rm ref}}(c_{\rm ref}) - w^\star_{S_{\rm cl}}(c_{\rm cl}) \|
\le
\rho_{\min} \|p(c_{\rm ref}) - p(c_{\rm cl}) \| + \delta + 2\sqrt{\rho_{\min}\delta}.
$$

\end{proof}

\begin{proof}{Proof of \cref{lem:uniform-convergence}}
By \cref{lem:pred_lipschitz}, the SPO+ loss is $(2D_S)$-Lipschitz in the prediction
argument. The result follows from the vector contraction inequality for the Rademacher complexity
and standard concentration \citep{liu2021risk}.
\end{proof}

\begin{proof}{Proof of \cref{thm:federation-gain}}
\textbf{Step 1 (Local ERM bound).}
Apply \cref{lem:uniform-convergence} to client $j$ with confidence $\delta/2$.
For client $j$, with probability $\ge 1-\delta/2$,
\[
\sup_{g\in\mathcal{H}} \bigl|R_j(g) - \widehat{R}_j(g)\bigr|
\le \varepsilon_j(n_j,\delta/2).
\]
The standard ERM argument gives:
\begin{align*}
R_j(\hat g_j^{\rm loc})
&\le \widehat{R}_j(\hat g_j^{\rm loc}) + \varepsilon_j \\
&\le \widehat{R}_j(g_j^\star) + \varepsilon_j
&&\text{(by optimality of $\hat g_j^{\rm loc}$)} \\
&\le R_j(g_j^\star) + 2\varepsilon_j.
\end{align*}

\textbf{Step 2 (Federated ERM bound on mixture risk).}
Apply \cref{lem:uniform-convergence} to the mixture distribution with sample size $N$
and confidence $\delta/2$. With probability $\ge 1-\delta/2$,
\[
\sup_{g\in\mathcal{H}} \bigl|R_{\rm mix}(g) - \widehat{R}_{\rm mix}(g)\bigr|
\le \varepsilon_{\rm mix}(N,\delta/2).
\]
By the same ERM argument:
\[
R_{\rm mix}(\hat g^{\rm fed}) \le R_{\rm mix}(g_{\rm mix}^\star) + 2\varepsilon_{\rm mix}.
\]

\textbf{Step 3 (Relating client risk to mixture risk).}
By \cref{def:discrepancy},
\[
R_j(g) \le R_{\rm mix}(g) + \Delta_j
,\quad
R_{\rm mix}(g) \le R_j(g) + \Delta_j
\]
for all $g\in\mathcal{H}$. Therefore:
\begin{align*}
R_j(\hat g^{\rm fed})
&\le R_{\rm mix}(\hat g^{\rm fed}) + \Delta_j \\
&\le R_{\rm mix}(g_{\rm mix}^\star) + 2\varepsilon_{\rm mix} + \Delta_j \\
&\le R_{\rm mix}(g_j^\star) + 2\varepsilon_{\rm mix} + \Delta_j
 \\
&\le R_j(g_j^\star) + 2\varepsilon_{\rm mix} + 2\Delta_j,
\end{align*}
by optimality of $g_{\rm mix}^\star$, which completes the proof.
\end{proof}

\begin{proof}{Proof of \cref{prop:delta-bound}}
For any $g\in\mathcal{H}$, using the shared-prediction specialization of \cref{thm:main}
(setting $\hat c = g(x)$ eliminates the Lipschitz prediction-shift terms),
\[
\bigl|\ell_{S_i}(g(x),c^{(i)}) - \ell_{S_j}(g(x),c^{(j)})\bigr|
\;\le\; \mathcal{H}_{ij}(g(x)).
\]
Taking expectations and using the definition of $R_{\rm mix}$,
$$
\bigl|R_j(g) - R_{\rm mix}(g)\bigr|
= \Bigl|\sum_{i=1}^m \alpha_i \bigl(R_j(g) - R_i(g)\bigr)\Bigr|
\le \sum_{i=1}^m \alpha_i \mathbb{E}\bigl[\mathcal{H}_{ij}(g(x))\bigr].
$$
Taking the supremum over $g\in\mathcal{H}$ yields the result.

\end{proof}

\section{Additional Results}
\label{appsec:additional-results}
\subsection{Connection to estimation-based federation gain}
\label{app:extensions}
We now show that controlling prediction error controls SPO+ excess risk, so existing minimax estimation results immediately imply corresponding guarantees for SPO+ federation gain.

\begin{lemma}[Prediction heterogeneity via Lipschitzness of $\ell_{\mathrm{SPO}+}$]
\label{lem:pred_lipschitz}
Let $S \subset \mathbb{R}^d$ be a bounded feasible region with diameter $D_S < \infty$.
Fix any realized cost vector $c \in \mathbb{R}^d$. Then the SPO+ loss is Lipschitz
in the prediction argument: for all $\hat c_1,\hat c_2 \in \mathbb{R}^d$,
\[
\bigl|\ell_{\mathrm{SPO}+}^{S}(\hat c_1,c)-\ell_{\mathrm{SPO}+}^{S}(\hat c_2,c)\bigr|
\;\le\; 2D_S \,\|\hat c_1-\hat c_2\|.
\]
\end{lemma}
\begin{proof}{Proof of \cref{lem:pred_lipschitz}}
For fixed $c$, the function $\hat c \mapsto \ell_{\mathrm{SPO}+}^{S}(\hat c,c)$ is convex
but generally non-smooth. A valid subgradient at $\hat c$ is 
\[
g(\hat c;c) \;=\; 2\bigl(w_S^\star(c) - w_S^\star(2\hat c - c)\bigr),
\]
where $w_S^\star(u) \in \arg\min_{w\in S} u^\top w$ denotes an optimal solution of the
downstream problem under objective $u$ (choose any selector if non-unique) \citep{elmachtoub2022smart}.
Since both $w_S^\star(c)$ and $w_S^\star(2\hat c-c)$ lie in $S$, we have
\[
\|g(\hat c;c)\|
\;=\; 2\|w_S^\star(c) - w_S^\star(2\hat c-c)\|
\;\le\; 2D_S.
\]
A convex function whose subgradients are norm bounded by $L$ is $L$-Lipschitz
in the dual norm, hence the claim follows with $L=2D_S$ (Lemma 2.6, \citep{shalev2012online}).
\end{proof}

\begin{corollary}[From estimation gain to SPO+ gain]
\label{cor:estimation-to-spo}
\citep{ho2022risk}. Let $\hat f$ be a predictor (local or federated) and $f_j^\star$ be the optimal predictor
for client $j$. By \cref{lem:pred_lipschitz},
$$
\mathbb{E}\bigl[R_j(\hat f) - R_j(f_j^\star)\bigr]
\;\le\; 2D_j \cdot \mathbb{E}\bigl[\|\hat f(x) - f_j^\star(x) \|\bigr]
\;\le\; 2D_j \sqrt{\mathbb{E}\bigl[\|\hat f(x) - f_j^\star(x) \|^2\bigr]}.
$$
Consequently, if the squared estimation error satisfies
$\mathbb{E}[\|\hat f - f_j^\star \|^2] \le R_j^{\rm (est)}$, then
\[
\mathbb{E}\bigl[R_j(\hat f) - R_j(f_j^\star)\bigr] \;\le\; 2D_j\sqrt{R_j^{\rm (est)}}.
\]
Therefore, if $\mathrm{FG}_j^{\rm est}$ denotes the ratio of local to federated minimax
estimation risks, then
\[
\mathrm{FG}_{j,\mathrm{bound}}^{\mathrm{SPO}+} \;\gtrsim\; \sqrt{\mathrm{FG}_j^{\rm est}}.
\]
\end{corollary}
\begin{proof}{Proof of \cref{cor:estimation-to-spo}}
The first inequality follows from \cref{lem:pred_lipschitz} and Jensen's inequality.
The second uses Cauchy--Schwarz. The gain relationship follows by taking ratios of the
local and federated bounds.
\end{proof}

\cref{cor:estimation-to-spo} tells us that by defining a ratio of risk upper bounds,
\begin{dmath*}
\mathrm{FG}_{j,\mathrm{bound}}^{\mathrm{SPO}+} = \frac{\text{local SPO+ risk upper bound}}{\text{fed. SPO+ risk upper bound}},
\end{dmath*}
we can relate the bounds to an upper bound of the squared estimation risk, here the minimax risk. We use $\gtrsim$ to account for potential constants in the minimax risks, but the key takeaway is that SPO+ federated gain scales with a rate equivalent to the square root of the MSE minimax risk ratio. In other words, controlling prediction error controls SPO+ excess risk, and a factor of $k$ improvement in minimax MSE translates to a factor of $\sqrt{k}$ improvement in the SPO+ risk bound.

In \citet{su2023non}, the authors provide a minimax estimation gain bound under a subspace heterogeneity model. Plugging their result into \cref{cor:estimation-to-spo} provides a lower bound as a function of local sample size, total sample size, ambient dimension, noise level, and a heterogeneity term $\Gamma$. We defer the full expression to \cref{rem:su-subspace} in Appendix \cref{appsec:additional-results}.

\begin{remark}[Interpretation of Bound Ratios]
$\mathrm{FG}_{j,\mathrm{cert}}$ is a PAC certified gain: it compares our high-probability excess-risk upper bounds (local vs federated + heterogeneity penalty). In contrast, $\mathrm{FG}_{j,\mathrm{bound}}^{\mathrm{SPO}+}$ comes from relating expectations of SPO+ risk to prediction estimation error and then using the minimax estimation gain comparisons, i.e., it should be interpreted as a scaling measure, and not as a certificate of the same type.
\end{remark}

\begin{remark}[Instantiation with Su et al.'s subspace model]
\label{rem:su-subspace}
In the subspace model of \citep{su2023non}, the minimax estimation gain satisfies
\[
\mathrm{FG}_j^{\rm est} \;\ge\;
\frac{c_1}{\kappa}\cdot
\frac{(\sigma^2 d/n_j \wedge B^2) + B^2(1-n_j/d,0)_+}{\sigma^2 d/N + \Gamma^2},
\]
where $\Gamma$ captures cross-client heterogeneity (and where $a \wedge b \coloneqq \min\{a,b\}.$) Plugging into \cref{cor:estimation-to-spo}:
$$
\mathrm{FG}_{j,\mathrm{bound}}^{\mathrm{SPO}+} \;\gtrsim\;
\sqrt{\frac{c_1}{\kappa}}\cdot
\sqrt{\frac{(\sigma^2 d/n_j \wedge B^2) +B^2(1-n_j/d,0)_+}{\sigma^2 d/N + \Gamma^2}}.
$$
\end{remark}
\subsection{Extension to SPO}

The SPO+ loss is a surrogate for the true SPO loss. \citet{liu2021risk} establish a calibration relationship that allows us to convert SPO+ bounds to SPO bounds for the federated case.

\begin{definition}[SPO regret]
The SPO regret of predictor $g$ on client $j$ is
\[
R_j^{\rm SPO}(g) := \mathbb{E}_{(x,c)\sim P_j}\bigl[c^\top w_{S_j}^\star(g(x)) - z_{S_j}^\star(c)\bigr].
\]
\end{definition}

\begin{theorem}[Calibration]
\label{thm:calibration}

{\citep{liu2021risk}}

\begin{enumerate}[label=(\roman*)]
\item \textbf{General (polyhedral) case:} There exists a calibration function $\psi$ such that
\[
R_j^{\rm SPO}(g) - R_j^{\rm SPO}(g_j^\star)
\;\le\; \psi\bigl(R_j(g) - R_j(g_j^\star)\bigr),
\]
with $\psi(\varepsilon) = O(\sqrt{\varepsilon})$ in the polyhedral case. This yields SPO excess
risk of order $O(n^{-1/4})$.

\item \textbf{Strongly convex feasible set:} Under \cref{ass:strong-convex-set},
the calibration becomes linear: $\psi(\varepsilon) = O(\varepsilon)$. This yields SPO excess
risk of order $O(n^{-1/2})$.
\end{enumerate}
\end{theorem}

\begin{corollary}[SPO federation gain]
\label{cor:spo-gain}
Comparing the local and federated SPO upper bounds gives us:
\begin{enumerate}[label=(\roman*)]
\item \textbf{Polyhedral case:}
\[
\mathrm{FG}_{j,\mathrm{bound}}^{\rm SPO}
\;\approx\; \sqrt{\mathrm{FG}_{j,\mathrm{bound}}^{\mathrm{SPO}+}}
\;\approx\; \bigl(\mathrm{FG}_j^{\rm est}\bigr)^{1/4}.
\]

\item \textbf{Strongly convex case:}
\[
\mathrm{FG}_{j,\mathrm{bound}}^{\rm SPO}
\;\approx\; \mathrm{FG}_{j,\mathrm{bound}}^{\mathrm{SPO}+}
\;\approx\; \bigl(\mathrm{FG}_j^{\rm est}\bigr)^{1/2}.
\]
\end{enumerate}

\end{corollary}


Table \ref{tab:bounds_summary} summarizes the chain of conversions. Strong convexity of the feasible set removes one
square root in the SPO$\to$SPO+ calibration step by making the optimizer stable
(\cref{cor:thm1-sc}).

\begin{table}[t]
\begin{center}
\caption{Summary of gain scaling}\label{tab:bounds_summary}
\begin{tabular}{lcc}
\hline
\textbf{Quantity} & \textbf{Polyhedral} & \textbf{Strongly Convex} \\
\hline
Estimation gain & $\mathrm{FG}_j^{\rm est}$ & $\mathrm{FG}_j^{\rm est}$ \\
SPO+ bound gain & $(\mathrm{FG}_j^{\rm est})^{1/2}$ & $(\mathrm{FG}_j^{\rm est})^{1/2}$ \\
SPO bound gain & $(\mathrm{FG}_j^{\rm est})^{1/4}$ & $(\mathrm{FG}_j^{\rm est})^{1/2}$ \\
\hline
\end{tabular}
\end{center}
\end{table}

\subsection{SPO+ as a Bregman Divergence}
\begin{proposition}[Strongly convex set calibration]
\label{prop:calibration}
Let $D_f(x, y) = f(x) - f(y) - \langle \nabla f(y),\, x - y \rangle$ denote the Bregman divergence associated with a convex function~$f$. Under \cref{ass:strong-convex-set}, for any hypothesis $f \in \mathcal{H}$ and any client~$j$,
$$
  R_{\mathrm{SPO+}}(\hat{f}) - R_{\mathrm{SPO+}}(f^\star_j)
  \leq\ \frac{2}{\rho}\ \mathrm{MSE}_j(\hat{f}),
$$
where $f^\star_j \in \arg\min_{f \in \mathcal{H}} R_{\mathrm{SPO+}}(f)$.
\end{proposition}
\begin{proof}{Proof of \cref{prop:calibration}}
The SPO+ loss can be represented as
$\ell_{\mathrm{SPO+}}(\hat{c}, c)
  = D_{\xi_S}\bigl(c - 2\hat{c}, -c\bigr).$
Under \cref{ass:strong-convex-set}, the support function~$\xi_S$ is $(1/\rho)$-smooth, so by the descent lemma,
$$
  D_{\xi_S}(x,\, y)
  \leq \frac{1}{2\rho}\,\|x - y\|^2.
$$
Substituting $x = c - 2\hat{c}$, $y = -c$ yields
$\ell_{\mathrm{SPO+}}(\hat{c},\, c)
  \leq \frac{2}{\rho}\,\|\hat{c} - c\|^2.$
Taking expectations over $(X,C)$ as in \cref{sec:fed-gain},
$$
  R_{\mathrm{SPO+}}(f)
  \leq\ \frac{2}{\rho}\mathrm{MSE}_j(f).
$$
Subtracting $R_{\mathrm{SPO+}}(f^\star_j)$ from both sides and noting that the MSE term for $f^\star_j$ is non-negative gives the result.
\end{proof}

\cref{prop:calibration} shows that when the feasible set $S$ is strongly convex, SPO+ behaves like a smooth quadratic loss in the cost-prediction error. This is a precursor for our experimental results: \emph{clients with strongly convex feasible sets benefit more from federation and are less sensitive to heterogeneity than those with polyhedral sets.}

Another takeaway is the geometric intuition: as a Bregman divergence of the support function, strong convexity and therefore smoothness of the support function implies SPO+ is bounded above by a quadratic. 
It is important to note that \cref{prop:calibration} is an excess risk bound by raw MSE bound, not excess MSE. While this is still useful as an upper bound, to get a true excess to excess bound we would need a stronger realizability assumption, or a no Bayes noise assumption on the mapping $f^{\star}(\mathcal{X})\mapsto\mathcal{C}$.

\FloatBarrier
\subsection{Additional Figures \& Tables}
\label{appsubsec:additional-figures}

In \cref{tab:aggregate-performance-app} we present detailed results on each method. We also include an Oracle method here for the interpolation methods, to quantify the validation error and statistical loss introduced by Interp-DFFL. To do so, we generate an additional 2000 validation samples for each client, and do not perform a data split. That way clients do not suffer a lower sample size or a significant validation to test transfer mismatch. The results are also shown in boxplots in \cref{fig:empirical-oracle}. Overall the gain is apparent but modest. Finally, in \cref{tab:data-rich-scarce} we have detailed comparisons between data rich and data scarce clients in the synthetic experiments across method and experimental domain. 

\begin{figure}
    \centering
    \includegraphics[width=0.5\linewidth]{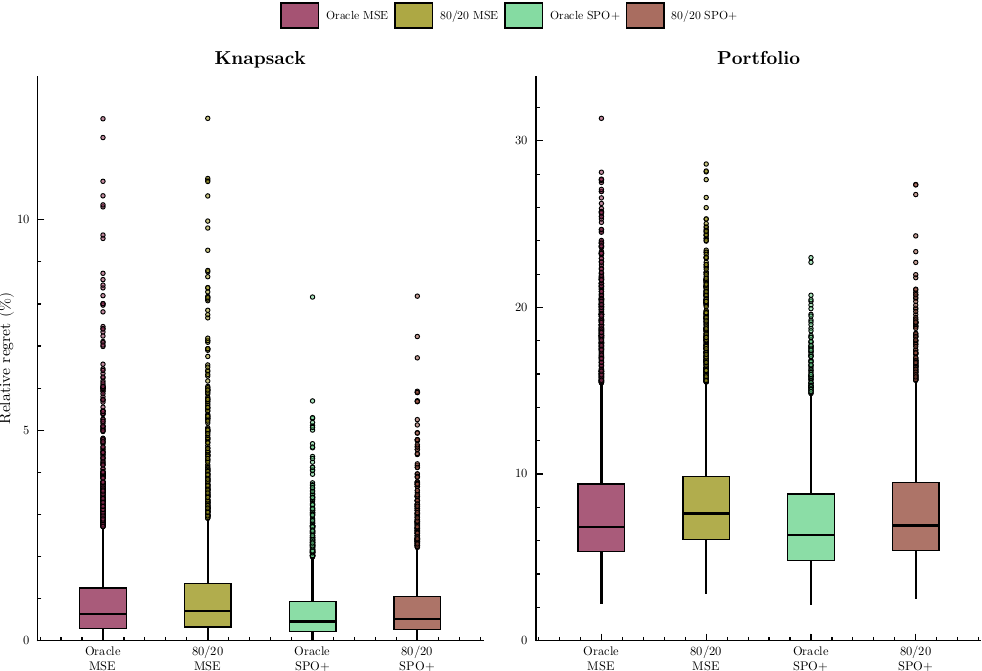}
    \caption{The impact of validation error and smaller sample sizes on the InterpDFFL and InterpMSE methods. }
    \label{fig:empirical-oracle}
\end{figure}

\begin{table}[t]
\centering
\small
\caption{Aggregate relative regret (\%) summary by experiment and method. Statistics are averaged over seeds and reported with inter-seed standard deviation. "E" denotes the experiment, where K: knapsack and P: portfolio. Lower is better $\downarrow$.}
\label{tab:aggregate-performance-app}
\begin{tabular}{lllllll}
\toprule
E & method & mean & median & p75 & p90 & max \\
\midrule
K & Local                     & $1.21 \pm 0.15$ & $1.10 \pm 0.14$ & $1.73 \pm 0.20$ & $2.41 \pm 0.31$ & $6.34 \pm 1.45$ \\
K & Federated                 & $1.37 \pm 0.18$ & $0.65 \pm 0.15$ & $1.83 \pm 0.31$ & $3.73 \pm 0.52$ & $11.70 \pm 0.83$ \\
K & Interp-DFFL(MSE)          & $1.10 \pm 0.14$ & $0.71 \pm 0.10$ & $1.37 \pm 0.16$ & $2.48 \pm 0.30$ & $10.36 \pm 2.04$ \\
K & Interp-DFFL(SPO+)         & $0.77 \pm 0.12$ & $0.51 \pm 0.08$ & $1.04 \pm 0.16$ & $1.78 \pm 0.29$ & $6.31 \pm 1.37$ \\
K & Oracle Interp-DFFL(MSE)   & $1.02 \pm 0.13$ & $0.64 \pm 0.11$ & $1.25 \pm 0.18$ & $2.35 \pm 0.29$ & $10.28 \pm 2.46$ \\
K & Oracle Interp-DFFL(SPO+)  & $0.68 \pm 0.10$ & $0.46 \pm 0.08$ & $0.93 \pm 0.15$ & $1.53 \pm 0.25$ & $5.48 \pm 1.60$ \\
P & Local                     & $11.54 \pm 0.49$ & $11.29 \pm 0.23$ & $13.84 \pm 0.51$ & $17.19 \pm 1.54$ & $33.57 \pm 3.01$ \\
P & Federated                 & $9.84 \pm 0.32$ & $7.02 \pm 0.37$ & $10.76 \pm 0.45$ & $22.00 \pm 0.96$ & $31.29 \pm 2.88$ \\
P & Interp-DFFL(MSE)          & $8.89 \pm 0.18$ & $7.65 \pm 0.32$ & $9.87 \pm 0.25$ & $15.72 \pm 0.61$ & $25.81 \pm 2.49$ \\
P & Interp-DFFL(SPO+)         & $7.84 \pm 0.26$ & $6.86 \pm 0.31$ & $9.46 \pm 0.27$ & $12.69 \pm 0.55$ & $21.78 \pm 3.39$ \\
P & Oracle Interp-DFFL(MSE)   & $8.40 \pm 0.11$ & $6.85 \pm 0.29$ & $9.29 \pm 0.34$ & $15.83 \pm 0.26$ & $27.97 \pm 1.97$ \\
P & Oracle Interp-DFFL(SPO+)  & $7.13 \pm 0.23$ & $6.33 \pm 0.26$ & $8.76 \pm 0.26$ & $11.69 \pm 0.46$ & $20.33 \pm 2.53$ \\
\bottomrule
\end{tabular}
\end{table}

\begin{table}[t]
\centering
\scriptsize
\caption{Data-rich and data-scarce clients under the imbalanced training regime. Statistics are averaged over seeds and reported with inter-seed standard deviation.}
\label{tab:data-rich-scarce}
\begin{tabular}{llllllll}
\toprule
domain & client group & method & mean regret (\%) & median regret (\%) & mean $\lambda$ & median $\lambda$ \\
\midrule
Knapsack & High-sample clients & Local & $0.45 \pm 0.06$ & $0.42 \pm 0.04$ & -- & -- \\
Knapsack & High-sample clients & Federated & $1.29 \pm 0.23$ & $0.52 \pm 0.07$ & -- & -- \\
Knapsack & High-sample clients & Interp-DFFL(MSE) & $0.88 \pm 0.13$ & $0.46 \pm 0.09$ & $0.86 \pm 0.06$ & $0.95 \pm 0.05$ \\
Knapsack & High-sample clients & Interp-DFFL(SPO+) & $0.38 \pm 0.04$ & $0.34 \pm 0.05$ & $0.43 \pm 0.03$ & $0.41 \pm 0.04$ \\
Knapsack & Low-sample clients & Local & $1.95 \pm 0.41$ & $1.89 \pm 0.31$ & -- & -- \\
Knapsack & Low-sample clients & Federated & $1.46 \pm 0.37$ & $0.64 \pm 0.30$ & -- & -- \\
Knapsack & Low-sample clients & Interp-DFFL(MSE) & $1.22 \pm 0.30$ & $0.90 \pm 0.23$ & $0.72 \pm 0.05$ & $0.78 \pm 0.06$ \\
Knapsack & Low-sample clients & Interp-DFFL(SPO+) & $0.99 \pm 0.28$ & $0.62 \pm 0.23$ & $0.67 \pm 0.07$ & $0.78 \pm 0.08$ \\
Portfolio & High-sample clients & Local & $6.09 \pm 0.13$ & $5.95 \pm 0.09$ & -- & -- \\
Portfolio & High-sample clients & Federated & $8.81 \pm 0.51$ & $5.14 \pm 0.23$ & -- & -- \\
Portfolio & High-sample clients & Interp-DFFL(MSE) & $6.32 \pm 0.33$ & $5.91 \pm 0.19$ & $0.73 \pm 0.06$ & $0.83 \pm 0.08$ \\
Portfolio & High-sample clients & Interp-DFFL(SPO+) & $5.51 \pm 0.15$ & $5.28 \pm 0.19$ & $0.67 \pm 0.03$ & $0.72 \pm 0.04$ \\
Portfolio & Low-sample clients & Local & $16.46 \pm 1.47$ & $15.49 \pm 1.41$ & -- & -- \\
Portfolio & Low-sample clients & Federated & $9.88 \pm 0.59$ & $5.98 \pm 0.54$ & -- & -- \\
Portfolio & Low-sample clients & Interp-DFFL(MSE) & $9.39 \pm 0.59$ & $8.10 \pm 0.59$ & $0.84 \pm 0.06$ & $0.88 \pm 0.08$ \\
Portfolio & Low-sample clients & Interp-DFFL(SPO+) & $8.43 \pm 0.39$ & $6.84 \pm 0.45$ & $0.89 \pm 0.02$ & $0.93 \pm 0.03$ \\
\bottomrule
\end{tabular}
\end{table}

In \cref{fig:knapsack-slope-grid-homo,fig:portfolio-slope-grid-homo} (homogeneous data) and \cref{fig:knapsack-slope-grid-hetero,fig:portfolio-slope-grid-hetero} (heterogeneous data), each slope line corresponds to one client, and their relative regret when going from a federated model (left) to a local model (right). If the line slopes upwards, it means the client prefers a federated model in that configuration. In \cref{fig:knapsack-slope-grid-hetero,fig:portfolio-slope-grid-hetero} specifically, we also differentiate clients based on their data sample size; orange slope lines represent data scarce-clients, and blue slope lines represent data-rich clients.

\begin{figure}[t]
    \centering
    \includegraphics[width=\textwidth]{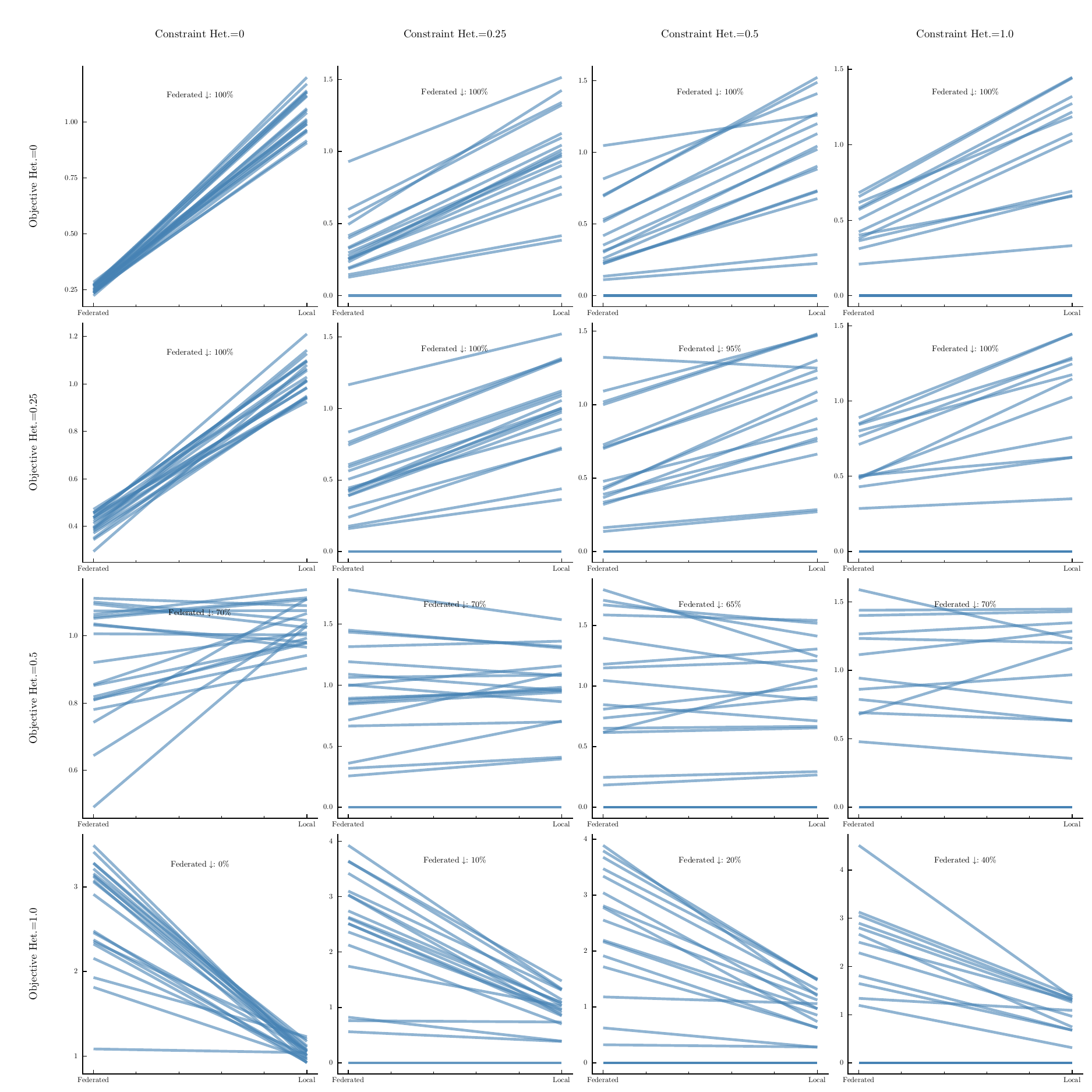}
    \caption{Local versus global regret per client in the \textbf{homogeneous data} setting for the \textbf{fractional knapsack} problem.}
    \label{fig:knapsack-slope-grid-homo}
\end{figure}
\begin{figure}[t]
    \centering
    \includegraphics[width=\textwidth]{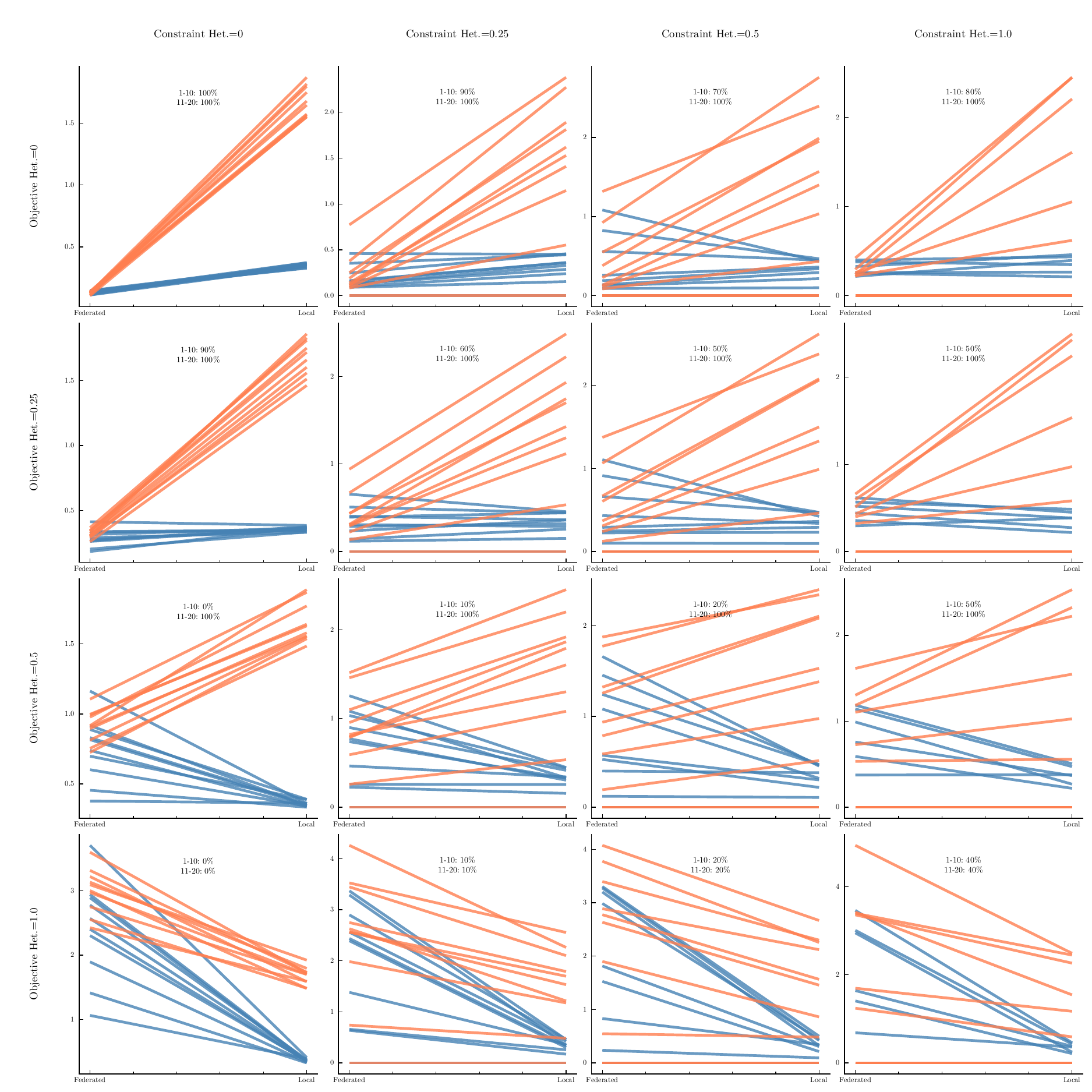}
    \caption{Local versus global regret per client in the \textbf{heterogeneous data} setting for the \textbf{fractional knapsack} problem for data-rich (blue) versus data-scarce clients (orange).}
    \label{fig:knapsack-slope-grid-hetero}
\end{figure}

\begin{figure}[t]
    \centering
    \includegraphics[width=\textwidth]{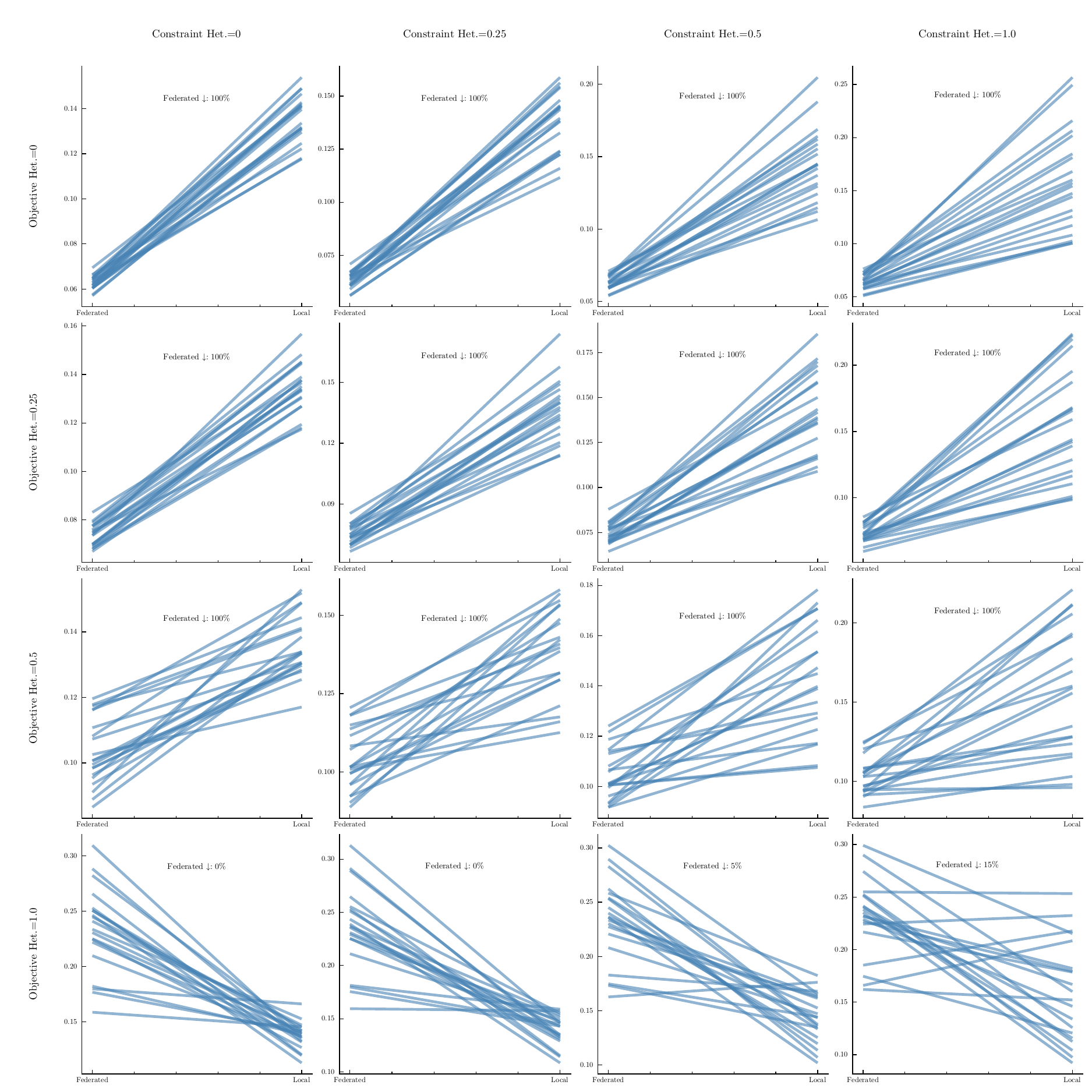}
    \caption{Local versus global regret per client in the \textbf{homogeneous data} setting for the \textbf{entropy portfolio} problem.}
    \label{fig:portfolio-slope-grid-homo}
\end{figure}
\begin{figure}[t]
    \centering
    \includegraphics[width=\textwidth]{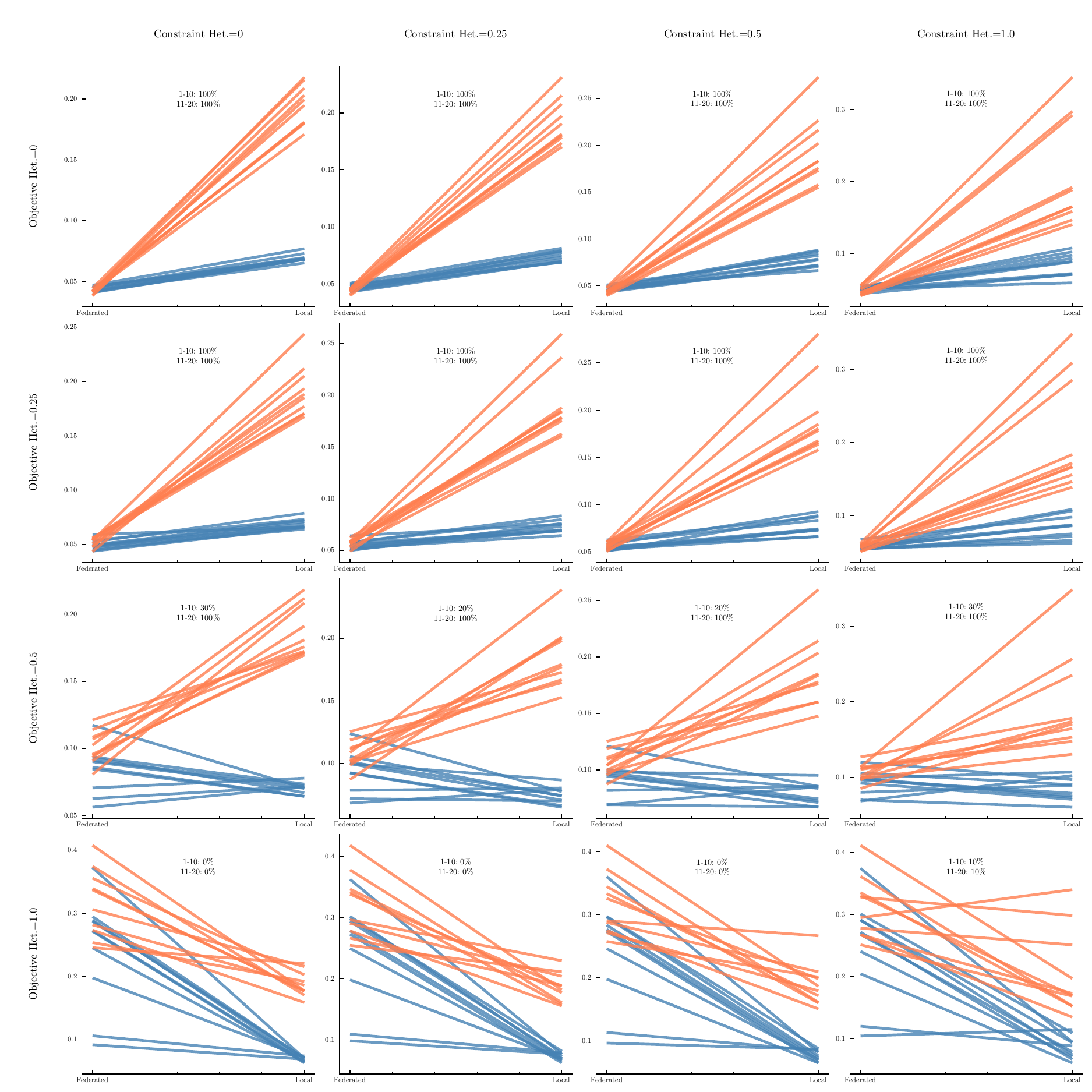}
    \caption{Local versus global regret per client in the \textbf{heterogeneous data} setting for the \textbf{entropy portfolio} problem for data-rich (blue) versus data-scarce clients (orange). }
    \label{fig:portfolio-slope-grid-hetero}
\end{figure}

\end{document}